\documentclass{article}
\usepackage[utf8]{inputenc}
\usepackage[margin=1in]{geometry}          
\usepackage{graphicx}
\usepackage{blkarray}
\usepackage{tikz}
\usepackage[round]{natbib}
\usepackage{float}
\usepackage{enumitem}
\usepackage{amsmath, amsthm, amssymb}
\allowdisplaybreaks
\usepackage{setspace}
\usepackage{mathtools}
\usepackage[colorlinks=true,linkcolor=blue,citecolor=blue,urlcolor=blue]{hyperref}
\usepackage{cleveref}
\usepackage{bm}

\DeclarePairedDelimiter\norm{\lVert}{\rVert}

\title{Treatment effect estimation under convergent network interference}
\author{
Bryan Park\\
\texttt{bryan314@stanford.edu}
\and
Stefan Wager\\
\texttt{swager@stanford.edu}
}
\date{Stanford University \\ \today}

\begin{document}

\maketitle

\newtheorem{theorem}{Theorem}
\newtheorem{problem}{Problem}
\newtheorem{lemma}{Lemma}
\newtheorem{definition}{Definition}
\newtheorem{example}{Example}
\newtheorem{observation}{Observation}
\newtheorem*{answer}{Answer}
\newtheorem{corollary}[lemma]{Corollary}
\newtheorem{proposition}[lemma]{Proposition}
\newtheorem{condition}{Condition}
\newtheorem*{remark}{Remark}

\newcommand{\Var}{\operatorname{Var}}
\newcommand{\Cov}{\operatorname{Cov}}
\newcommand{\HT}{\operatorname{HT}}
\newcommand{\Haj}{\operatorname{HAJ}}
\newcommand{\C}{\mathcal{C}}
\newcommand{\E}{\mathbb{E}}
\newcommand{\N}{\mathbb{N}}
\newcommand{\G}{\mathcal{G}}
\newcommand{\F}{\mathcal{F}}
\newcommand{\Z}{\mathbb{Z}}
\newcommand{\Q}{\mathbb{Q}}
\newcommand{\FX}{\mathcal{F}^\mathbf{X}}
\renewcommand{\L}{\mathcal{L}}
\renewcommand{\S}{\mathcal{S}}
\newcommand{\T}{\mathsf{T}}
\renewcommand{\P}{\mathbb{P}}
\newcommand{\B}{\mathcal{B}}
\newcommand{\A}{\mathcal{A}}
\newcommand{\V}{\mathcal{V}}
\newcommand{\D}{\mathcal{D}}
\newcommand{\R}{\mathbb{R}}
\newcommand{\lims}[1][]{\xrightarrow[#1]{\rho_n}}
\newcommand{\limd}[1][]{\xrightarrow[#1]{d}}
\newcommand{\limp}[1][]{\xrightarrow[#1]{p}}
\newcommand{\limas}[1][]{\xrightarrow[#1]{\mathrm{a.s.}}}
\makeatletter
\newcommand{\Mod}[1]{\ (\mathrm{mod}\ #1)}
\newcommand{\ed}{\,{\buildrel d \over =}\,}

\begin{abstract}
Under network interference, a unit's observed outcome depends on the treatment assignment of its neighboring units in an exposure graph. Existing design-based asymptotic theory typically considers local interference by restricting neighborhood sizes in the exposure graph. Such methods do not apply to dense exposure graphs, so prior work has often adopted a superpopulation approach instead, imposing regularity through random-graph models. In this paper, we introduce a notion of convergence for a sequence of finite populations under anonymous interference. Building on the graph limit framework of Lov\'{a}sz and Szegedy, we show that large-scale geometry of the exposure graph can provide a source of regularity beyond sparsity assumptions or random-graph modeling. Under Bernoulli assignment, our convergence notion yields asymptotic normality of standard estimators for the average direct effect, even on dense, non-random exposure graphs. As a special case, graphon-based random-graph models studied in prior work generate finite populations that converge in our sense. Under these models, graph randomness generates exposure graphs with stable large-scale geometry, while first-order uncertainty in average direct effect estimation is driven by treatment assignment.
\end{abstract}

\section{Introduction}\label{intro_sec}

Design-based causal inference studies a fixed, finite population of units and potential outcomes, using treatment assignment as the sole source of randomness \citep{neyman1923applications}. Under the stable unit treatment value assumption (SUTVA) \citep{rubin1980sutva}, each unit’s observed outcome depends only on its own treatment. Under interference, however, a unit's observed outcome depends on its exposure, which may be the full treatment vector or a lower-dimensional summary \citep{Hudgens,Manski,Aronow}. Thus, randomization induces a joint distribution over the units' exposures, and the realized exposures then determine the observed outcomes. From this perspective, the relevant design-side regularity for asymptotic theory concerns the joint behavior of exposures as the number of units grows.

Under network (or neighborhood) interference, a unit's exposure depends on its own treatment along with the treatment assignment of its neighboring units in an exposure graph. Existing design-based theory typically restricts neighborhood sizes in the exposure graph, ensuring that dependence among exposures is local \citep{Savje,leung2022ani,gao2025network}. This approach, however, does not apply to dense exposure graphs, where a unit’s neighborhood overlaps with many other units’ neighborhoods. There has been some recent work studying treatment effect estimation under dense network interference via random-graph models  \citep{Wager,shirani2024cmp,bhattacharya2025meanfield}. While these results allow complex models of interference, their stochastic generative assumptions also move us away from the design-based paradigm for causal inference.


In this paper, we study design-based asymptotics for network interference on possibly dense exposure graphs, without relying on random-graph assumptions to justify our analysis. We focus on anonymous interference, where each unit’s exposure consists of its own treatment and the fraction of its treated neighbors \citep{Manski}. In this setting, the exposure graph directly determines the joint distribution of exposures induced by the randomization. Our key idea is that large-scale geometry of the exposure graph can provide a source of regularity beyond sparsity assumptions or random-graph modeling. Formally, we introduce a notion of convergence for a sequence of finite populations under anonymous interference. Our notion extends the graph limit framework of \cite{Lovasz}, regulating not only the exposure graph but also how each unit's outcome depends on its possible exposures.

Under the Bernoulli design, our notion of convergence yields $\sqrt{n}$-asymptotic normality of both the Horvitz--Thompson and H\'{a}jek estimators for the average direct effect, even on dense, non-random exposure graphs. As a special case, we show that the random-graph model of \cite{Wager} almost surely generates finite populations that converge in our sense. Under their model, graph randomness generates exposure graphs with stable large-scale geometry, while first-order uncertainty in average direct effect estimation is driven by treatment assignment.

\subsection{Related Work}

Our notion of convergence is motivated by \citet{Lovasz} who introduced the concept of graph limits for sequences of dense
graphs. They defined graph limits by requiring subgraph densities to converge. In particular, they showed that a symmetric, measurable
function $A : [0,1]^2 \to [0,1]$, now widely known as a graphon, captures the limiting subgraph densities. Next, \citet{graphon} established the
equivalence between convergence in subgraph densities and convergence in the cut metric. The cut metric
provides a formal distance measure between graphs that captures their large-scale structural similarities.
While the original theory was developed for dense graph sequences, subsequent work by \citet{key} and others extended this framework to sparse graph sequences via an $L^p$ theory of graph limits.

Recently, graphons have gained significant traction across a number of fields, ranging from economics and operations research to statistics. For instance, \citet{Parise} and \citet{Erol}
used graphons to study optimal interventions in large-scale network games and contagion processes. The aforementioned random-graph models \citep{Wager, shirani2024cmp, bhattacharya2025meanfield} include graphon models for interference structures. Crucially, however, these applications treat graphons as generative models, viewing observed finite networks
as random draws from an underlying graphon. To the best of our knowledge, our paper is the first to introduce graph limit theory into the causal inference literature as a tool for formulating deterministic regularity conditions that enable asymptotic theory under interference.

Broadly speaking, our paper fits into the literature on causal inference under interference, where a unit's observed outcome may depend on treatments assigned to other units \citep{Halloran}. We work under the exposure-mapping framework by assuming anonymous interference \citep{Hudgens,Manski,Aronow}. However, our causal estimand does not depend on any exposure mapping. We are interested in the average direct effect, defined as the average contrast between treated and untreated potential outcomes holding others’ treatments fixed \citep{Savje,Hu}.

Methodologically, our paper builds on the network interference literature which encompasses both design-based and superpopulation approaches. From a design-based perspective, \cite{Savje} built a general theory for sparse interference, even without imposing exposure mappings. In network settings, their theory implies that $\sqrt{n}$-consistency can be obtained if the maximum degree of the exposure graph is bounded. Central limit theorems were further proved by \citet{leung2022ani} and \citet{gao2025network} under approximate neighborhood interference, again with some restrictions on neighborhood sizes. In another line of work, \cite{Savje2} introduced a quantitative measure of interference that does not require sparsity. His result implies that under anonymous interference, $\sqrt{n}$-consistency can hold for dense exposure graphs if pairwise interference remains weak. More recently, \cite{Lu} developed a design-based theory under uniformly bounded neighborhood interference, providing central limit theorems for dense exposure graphs as well. Their framework is complementary to ours, as they impose outcome-side smoothness and influence restrictions rather than regulating the large-scale geometry of the exposure graph.

From a superpopulation perspective, \cite{Wager} modeled the exposure graph as a random draw from a graphon.  Under anonymous interference, they showed that the Horvitz--Thompson estimator and H\'{a}jek estimator are $\sqrt{n}$-asymptotically normal for the average direct effect, regardless of graph density. A key difference from the design-based approach is that such asymptotic normality relies on randomness from both treatment assignment and graph generation. In this paper, we provide a design-based interpretation of Li and Wager's random-graph asymptotics, based on the fact that their random-graph model generates convergent finite populations. In general, a growing literature has developed random-graph approaches to network interference. For instance, \citet{shirani2024cmp} and \citet{bhattacharya2025meanfield} use approximate message passing or mean-field structure to estimate effects under rich propagation and long-range dependence. Other complementary directions include \cite{Gao} on endogenous network formation, \cite{Fan} on regression adjustment, and \cite{Cattaneo} on robust inference under both outcome and treatment-assignment interference.

\section{Model and Preliminaries}
In Section \ref{problem_setting}, we formally describe our problem setup. In Sections \ref{SUTVA_sec} and \ref{Graph_sec}, we provide preliminary results for our notion of convergence for finite populations. Section \ref{SUTVA_sec} motivates a regularity condition for potential outcome functions by revisiting SUTVA. Section \ref{Graph_sec} introduces the graph limit framework of \cite{Lovasz} that describes large-scale geometry of graphs.

\subsection{Problem Setting}\label{problem_setting}

We study Bernoulli treatment assignment under anonymous interference. In order to perform asymptotic analysis, we embed the given population in an infinite sequence of growing finite populations \citep{Lin,lumley2010asymp}. For each population size $n$, let $[n]=\{1,2,\dots,n\}$ represent units. For each $i\in[n]$, we posit potential outcomes $\{Y_i(w):w\in\{0,1\}^n\}$ so that 
\begin{align}
    \{Y_i(w): i\in [n], w\in\{0,1\}^n\} \label{eq:general_fp}
\end{align}
describes the $n$th finite population. Let $W=(W_1,\dots,W_n)$ denote the random treatment vector, where $W_i \overset{\mathrm{i.i.d.}}{\sim} \mathrm{Ber}(\pi)$ and $\pi\in(0,1).$ For each unit $i\in[n]$, we let
$$
Y_i=Y_i(W)=Y_i(W_1,\dots,W_n)
$$
denote the observed outcome of unit $i$ under treatment assignment $W$.

Under anonymous interference, we assume that $Y_i(W_1,\dots,W_n)$ depends on $W$ only through unit $i$'s own treatment status and the fraction of its treated neighbors in an exposure graph. This allows a structural description of \eqref{eq:general_fp}. Formally, let $G_n$ denote a simple undirected exposure graph and write
$$
R_i(G_n)=\frac{\sum_{j\sim i} W_j}{d_i(G_n)},
\qquad
d_i(G_n)=\deg_{G_n}(i)\vee 1,
$$
where $j\sim i$ means that $i,j$ are neighbors in the graph $G_n$. We assume that
\[
Y_i(W_1,\dots,W_n)=f_i\bigl(W_i,R_i(G_n)\bigr)
\]
for some function $f_i\in\mathcal{F}$, where
\begin{align}
\mathcal{F}
=
\left\{
f:\{0,1\}\times[0,1]\to\mathbb{R}
\ \text{such that}\
|f^{(k)}(w,x)|\le C
\ \text{for } k\in\{0,1,2,3\}
\right\},\label{eq:tech_potential}
\end{align}
$C$ is a constant, and the derivatives are taken with respect to $x$. For each unit $i\in [n]$, we refer to $f_i$ as its potential outcome function and $(W_i,R_i(G_n))$ as its exposure.

We thus summarize the $n$th finite population \eqref{eq:general_fp} by the pair $(G_n,v_n)$, where $G_n$ also denotes the adjacency matrix of the exposure graph and $v_n:[n]\to\mathcal{F}$ is the map $i\mapsto f_i$ describing the vector $(f_1,\dots, f_n)$. The design-based perspective views $(G_n,v_n)$ as deterministic, whereas the superpopulation perspective views $(G_n,v_n)$ as random. Following \citet{Savje} and \citet{Hu}, we define the average direct effect (ADE) as
$$ \overline\tau_n(G_n,v_n) = \frac{1}{n}\sum_{i=1}^n \E_\pi\left[f_i(1,R_i(G_n))-f_i(0,R_i(G_n))\mid G_n, v_n\right].$$
We consider the Horvitz--Thompson (HT) estimator
$$  \hat\tau_n^{\HT}(G_n, v_n) = \frac{1}{n}\sum_{i=1}^n\left(\frac{W_i}{\pi}-\frac{1-W_i}{1-\pi}\right)Y_i$$
along with the H\'{a}jek estimator
$$  \hat\tau_n^{\Haj}(G_n, v_n) = \frac{\sum_{i=1}^n W_i Y_i}{\sum_{i=1}^n W_i} -  \frac{\sum_{i=1}^n (1-W_i) Y_i}{\sum_{i=1}^n (1 - W_i)}.$$
When the estimator and finite population are clear, we write $\hat\tau_n,\overline\tau_n$ for simplicity. As we adopt a design-based perspective, the fixed population sequence $\{(G_n,v_n)\}_{n=1}^\infty$ is the central object of our asymptotic framework. The next sections build towards a notion of convergence that captures structural regularity of such sequences.

\subsection{Finite Population Convergence under SUTVA}\label{SUTVA_sec}
We first recall the classical SUTVA setting, where the $n$th finite population is simply
$$\{(Y_i(0),Y_i(1)):i\in [n]\}.$$ Writing
$$a_n = (Y_i(0))_{i=1}^n, \qquad b_n = (Y_i(1))_{i=1}^n,$$ we identify the finite population with the pair $(a_n,b_n)$. In design-based
asymptotics, a fixed finite population is embedded in an imagined sequence of
growing finite populations $\{(a_n,b_n)\}_{n=1}^\infty$ \citep{Lin, lumley2010asymp}. The purpose of regularity conditions imposed on the sequence is not merely to prove a central limit theorem. Rather, these
conditions require the growing populations to be sufficiently similar to the given population, ensuring that the limiting approximation is informative.

A standard way to formalize this requirement is through the empirical
distribution of potential outcomes. Let $F_n$ denote the empirical distribution
of $\{(Y_i(0),Y_i(1))\}_{i=1}^n$, placing mass $1/n$ at each pair. Classical
finite-population asymptotics often assumes that
$$F_n\Rightarrow F$$
for some limiting distribution $F$ \citep[Example 12.2.2]{lehmann2022TSH}. This condition says that the joint empirical distribution of potential outcomes stabilizes along the sequence. Thus, the sequence
does not consist of arbitrary growing populations, and instead preserves relevant features for the treatment-effect estimator.

For our purposes, it is useful to express the same idea geometrically. With slight abuse of notation, we identify the vectors $a_n = (Y_i(0))_{i=1}^n$ and $b_n=(Y_i(1))_{i=1}^n$ as step functions on $[0,1]$ by letting 
$$a_n(t) = a_n(i) = Y_i(0), \qquad b_n(t) = b_n(i) = Y_i(1)$$ for $(i-1)/n < t < i/n.$ The boundary points can be assigned values arbitrarily. Next, for any permutation $\varphi_n\in S_n,$ let $a_n^\varphi$ denote the vector with $i$th entry $a_n^\varphi(i) = a_n(\varphi_n(i))$ where we suppress the dependency on $n$ in the subscript of $\varphi_n.$ Defining $b_n^\varphi$ similarly, we obtain the following notion of convergence that does not depend on the labeling of units.

\begin{definition}[Convergence of $(a_n,b_n)$]\label{fin_convg}
\normalfont
    Let $\{(a_n,b_n)\}_{n=1}^\infty$ be a sequence of finite populations under SUTVA and let $\ell_0,\ell_1$ be integrable functions on $[0,1].$ Let $\norm{\cdot}_1$ denote the $L^1$ norm for integrable functions on $[0,1]$. We say that 
    $$(a_n,b_n)\to (\ell_0,\ell_1)$$
    as $n\to\infty$ if there is a sequence of permutations $\varphi_n\in S_n$ so that $\norm{a_n^\varphi-\ell_0}_1\to 0$ and $\norm{b_n^\varphi-\ell_1}_1\to 0.$
\end{definition}

\begin{figure}[t]
    \centering
    \includegraphics[width=0.75\textwidth]{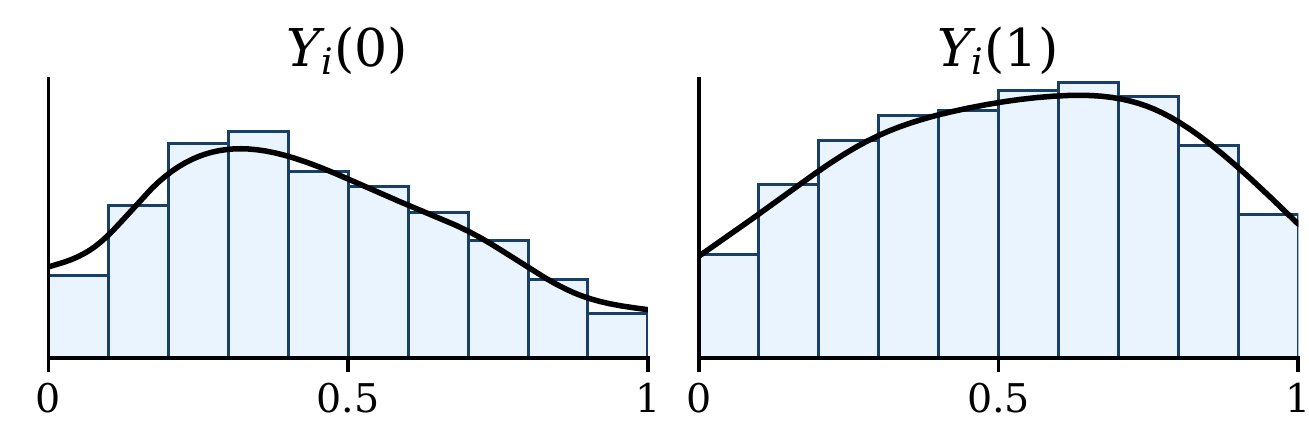}
    \caption{Finite population convergence under SUTVA. The step functions correspond to $n=10$, while the black curves describe the limiting shapes $\ell_0$ and $\ell_1$. Under the same labeling of units, we assume that each step function converges to the respective limiting shape as $n\to\infty$.}
    \label{fig:intuition}
\end{figure}

The key is that $a_n$ and $b_n$ converge under the same labelings. Figure \ref{fig:intuition} provides an illustration of Definition \ref{fin_convg}. While the convergence of $F_n$ captures the limiting distribution of potential outcome pairs, the convergence of $(a_n,b_n)$ encodes a limiting-shape view of finite-population asymptotics. A short calculation in the \hyperref[app]{Appendix} shows that convergence of $(a_n,b_n)$ implies convergence of $F_n$. More importantly, both perspectives express the same underlying principle: An asymptotic sequence is useful to the extent that its large populations resemble the finite population of study. We adopt this perspective when developing asymptotic regularity conditions under anonymous interference.

\subsection{Graph Limits and Large-Scale Geometry}\label{Graph_sec}

The previous section introduced a geometric notion of finite population convergence under SUTVA. Under network interference, this is no longer enough, as the exposure graph is also part of the finite-population geometry. In this section, we introduce graph limits to describe large-scale geometry of exposure graphs. Similar to how we embedded $a_n,b_n$ as step functions on $[0,1],$ we begin by representing finite graphs as step functions on $[0,1]^2$. 

A graphon is any symmetric, measurable function $A: [0,1]^2\to [0,1].$ Let $A_n$ denote the adjacency matrix of any simple undirected graph on $n$ vertices. We view $A_n$ as a graphon in the following way. Partition $[0,1]^2$ into an $n\times n$ grid of equal squares. For all points in the interior of the square $[(i-1)/n,i/n]\times [(j-1)/n,j/n],$ we let the graphon take on value $A_n(i,j)$. The remaining points on the grid lines form a null set of $[0,1]^2$, hence any values (respecting symmetry) may be assigned.  Next, for any adjacency matrix $A_n$ and permutation $\varphi_n\in S_n,$ let $A_n^\varphi$ denote the adjacency matrix
$A_n^\varphi(i,j)=A_n(\varphi_n(i),\varphi_n(j)).$ Allowing for such permutations enables a notion of graph convergence that is independent of vertex-labeling. 

While the original theory of \cite{Lovasz} focuses on graph limits of dense graphs, the framework we adopt allows sparse graph sequences through an appropriate rescaling. For this purpose, the correct limit objects turn out to be kernels, which are symmetric, integrable, nonnegative functions on $[0,1]^2$. In particular, any graphon is a kernel. Next, while $L^1$ convergence is natural for the potential outcome functions in Definition \ref{fin_convg}, it is too strong for adjacency matrices as it controls microscopic edge-by-edge discrepancies. Instead, we use the cut norm \citep{graphon} for integrable functions on $[0,1]^2,$ which is a key notion that captures large-scale edge structure.

\begin{definition}[Cut norm]\label{Cut norm}
\normalfont
    Given any integrable $A: [0,1]^2\to \R$, we define the cut norm as
    $$\norm{A}_\square = \sup_{S,T\subseteq [0,1]}\left|\int_{S\times T}A(u,v)\,d\lambda(u,v)\right|$$
    where $S,T$ are measurable and $\lambda$ is the Lebesgue measure on $[0,1]^2$.
\end{definition}

As an example, consider any two graphs $A_n,B_n$ on $n$ vertices. We see that
\begin{align*}
    \left|\frac{1}{n^2}\sum_{i\in S}\sum_{j\in T}\big(A_n(i,j)-B_n(i,j)\big)\right|\leq \norm{A_n-B_n}_\square
\end{align*}
for any $S,T\subseteq [n].$ In other words, $\norm{A_n-B_n}_\square$ uniformly controls discrepancies in the total edge mass between any two large vertex subsets. Indeed, only subsets of linear size matter due to the $n^{-2}$ normalization. Equipped with the cut norm, we use the following definition of graph convergence \citep{graphon, key}, which allows both vertex relabeling and graph rescaling.
\begin{definition}[Graph convergence at scale $\rho_n$]\label{Conv}
\normalfont
    Let $\{A_n\}_{n=1}^\infty$ be a sequence of simple undirected graphs and $\{\rho_n\}_{n=1}^\infty$ be a sequence of scale factors in $(0,1]$. Let $A$ be a kernel. We say that $$A_n\lims A$$ as $n\to\infty$ if there is a sequence of permutations $\varphi_n\in S_n$ so that $\norm{\rho_n^{-1}A_n^{\varphi}-A}_\square\to 0.$
\end{definition}

If $A_n\lims A,$ it follows from Definition \ref{Cut norm} that $\norm{A_n}_1/\rho_n\to \norm{A}_1$ where $\norm{\cdot}_1$ is the $L^1$ norm for integrable functions on $[0,1]^2.$ In particular, if $\norm{A}_1>0$, then $\rho_n$ must be of the same order as the edge densities $\norm{A_n}_1$. For this reason, instead of considering arbitrary scale factors, we consider $\rho_n=1$ (dense graphs) or $\rho_n\to 0$ and $n\rho_n\to\infty$ (sparse graphs). In the dense case $\rho_n=1,$ we assume that the limit $A$ is a graphon since entries of $A_n$ are in $\{0,1\}$. Moreover, note that our sparse case still requires growing average degrees.

\section{Finite Population Convergence under Interference}\label{Anonymous_sec}

Under anonymous interference, the sequence of populations is $\{(G_n,v_n)\}_{n=1}^\infty$ where $G_n$ is the exposure graph and $v_n:[n]\to \F$ describes the vector of potential outcome functions. Building on Sections \ref{SUTVA_sec} and \ref{Graph_sec}, we introduce a notion of convergence for $\{(G_n,v_n)\}_{n=1}^\infty$ that regulates the finite-population geometry and ensures that the asymptotic approximation is meaningful for the finite population of interest. The key idea is to require that the exposure graphs and potential outcome functions converge jointly, after a common relabeling of the units and an appropriate graph rescaling. The limiting object is $(\L,\ell),$ where $\L$ is a kernel and $\ell:[0,1]\to \F$ is measurable, meaning $(t,w,x)\mapsto \ell(t)(w,x)$ is measurable.

The convergence of exposure graphs is captured by Definition \ref{Conv}. To define convergence of potential outcome functions, we extend Section \ref{SUTVA_sec}. For each $(w,x)\in\{0,1\}\times [0,1]$, let $v_n(\cdot,w,x)$ denote the function on $[0,1]$ obtained by letting $v_n(t,w,x) = v_n(i)(w,x)= f_i(w,x)$ for $(i-1)/n < t < i/n.$ Indeed, the boundary points are assigned arbitrary values. Following Definition \ref{fin_convg}, we require each $v_n(\cdot,w,x)$ to converge to a limiting shape. Formally, the following is our notion of convergence for $\{(G_n,v_n)\}_{n=1}^\infty.$

\begin{definition}[Convergence of $(G_n,v_n)$ at scale $\rho_n$]\label{Mine}
\normalfont
     Let $\{(G_n, v_n)\}_{n=1}^\infty$ be a sequence of finite populations under anonymous interference. Let $\L$ be a kernel and $\ell: [0,1]\to \F$ be measurable. Finally, let $\{\rho_n\}_{n=1}^\infty$ be a sequence of scale factors in $(0,1]$. We say that $$(G_n,v_n)\lims (\L,\ell)$$ as $n\to\infty$ if there is a sequence of permutations $\varphi_n\in S_n$ so that 
    \begin{enumerate}[label=(\alph*)]
        \item $\norm{\rho_n^{-1}G_n^\varphi - \L}_\square\to 0,$
        \item $\norm{v_n^\varphi(\cdot, w, x) - \ell(\cdot, w, x)}_1\to 0$
    for each $(w,x)\in \{0,1\}\times [0,1],$ 
    \end{enumerate}
    where $v_n^\varphi(i) = v_n(\varphi_n(i))$ and $\ell(t,w,x)=\ell(t)(w,x).$ 
\end{definition}
Under SUTVA, we know that $G_n$ is the empty graph and $v_n(t ,w,x)$ does not depend on $x.$ Thus, Definition \ref{Mine} is a direct generalization of Definition \ref{fin_convg}, and a crucial feature is that the same sequence of permutations is used to align both the exposure graphs and the potential outcome functions. 

\subsection{Examples of Finite Population Convergence}\label{interpret_conv}

Before proceeding further, we provide some examples of Definition \ref{Mine}. We begin with mathematical examples, viewing $(G_n,v_n)$ as abstract objects. We then provide a latent-type interpretation in the context of causal inference. All proofs are given in the \hyperref[app]{Appendix}.

\subsubsection{Randomly Generated Examples}\label{stochastic_example}
As mentioned in Section \ref{intro_sec}, the random-graph model of \cite{Wager} is one example that produces such convergent sequences of populations. The design-based framework incorporates random graphs by conditioning on their realizations and only considering randomness from the treatment assignment. Let $\L$ be any kernel and $\ell:[0,1]\to \F$ be measurable. Let $U_i\overset{\text{i.i.d.}}{\sim}U[0,1]$ and let $\{\rho_n\}_{n=1}^\infty$ denote a sequence of scale factors. Following Li and Wager, we can define a random exposure graph $L_n$ and a random vector of potential outcome functions $\ell_n$ by 
\begin{align}
    L_n= G(n,\L,\rho_n),\qquad \ell_n=\left(\ell(U_1),\dots, \ell(U_n)\right). \label{random_pop}
\end{align} 
Concretely, $G(n,\L,\rho_n)$ is the random graph drawn from $\L$ at density $\rho_n$, generated as follows. Given $\{U_i\}_{i=1}^n$, for each unordered pair of vertices $\{i,j\}$, we add an edge independently with probability $\min\{1, \rho_n\cdot \L(U_i,U_j)\}.$

The random population $(L_n,\ell_n)$ can be viewed as a discrete approximation of $(\L,\ell)$ along random evaluation points $U_1,\dots, U_n.$ The following result shows that $(L_n,\ell_n)\lims (\L,\ell)$ almost surely.

\begin{lemma}\label{random_order}
Assume that either $\rho_n=1$ and $\L$ is a graphon, or that $\rho_n\to 0,$ $n\rho_n\to\infty,$ and $\L$ is a kernel. Next, let $\ell:[0,1]\to\F$ be measurable and let $\sigma_n$ denote the permutation that orders $U_1,\dots, U_n,$ namely $U_{\sigma_n(1)}\leq \cdots\leq U_{\sigma_n(n)}.$ Then, almost surely, we have  
        \begin{enumerate}[label=(\alph*)]
        \item $\norm{\rho_n^{-1}L_n^\sigma - \L}_\square\to 0$,
        \item $\norm{\ell_n^\sigma(\cdot,w,x) - \ell(\cdot,w,x)}_{1}\to 0$
    for each $(w,x)\in \{0,1\}\times [0,1].$ 
    \end{enumerate}
\end{lemma}
A special case of Lemma \ref{random_order} is the Erd\H{o}s--R\'{e}nyi graph $G(n,p),$ where each unordered pair of vertices is independently connected with probability $p.$ If $L_n=G(n,p)$ and $\L$ is the constant graphon with value $p,$ then we have $\norm{L_n^\sigma-\L}_\square\to 0$ almost surely.

\subsubsection{Rule-Based Examples}\label{rule_example}
Here, we show that convergence of finite populations extends well beyond the random-graph model considered in the previous section. One example is to obtain discrete approximations of Riemann integrable $(\L,\ell)$ along deterministic evaluation points $\{i/n\}_{i=1}^n.$ This example actually gives $L^1$ convergence, which is stronger than cut-norm convergence.
\begin{lemma}\label{det_point}
    Let $\L$ be a Riemann integrable graphon with range $\{0,1\}.$ Let $\ell(\cdot,w,x)$ be Riemann integrable for each $(w,x)\in\{0,1\}\times[0,1].$ Define $(G_n,v_n)$ by 
    $$G_n(i,j) = \L(i/n,j/n), \qquad v_n(i) = \ell(i/n).$$ Then, we have
    \begin{enumerate}[label=(\alph*)]
        \item $\norm{G_n - \L}_1\to 0$ and thus $\norm{G_n-\L}_\square\to 0$,
        \item $\norm{v_n(\cdot,w,x) - \ell(\cdot,w,x)}_{1}\to 0$
    for each $(w,x)\in \{0,1\}\times [0,1].$ 
    \end{enumerate}
\end{lemma}
Assuming $\norm{\L}_1>0,$ the above result implies that $G_n$ is dense. For instance, a special case of Lemma \ref{det_point} is the half graph $H_n$ given by $H_n(i,j) = 1_{i\neq j}\cdot 1_{i+j>n}$. Letting $\L$ denote the graphon $\L(u,v) = 1_{u\neq v}\cdot 1_{u+v>1},$ we have $\norm{H_n-\L}_1\to 0$. For sparse $G_n$, a similar result can be obtained by thinning a deterministic sequence of convergent dense graphs. Here, the convergence is no longer necessarily in the $L^1$ norm.

\begin{lemma}\label{thinning}
    Let $\{A_n\}_{n=1}^\infty$ denote a sequence of graphs satisfying $\norm{A_n-\L}_\square\to 0$ where $\L$ is a graphon. Let $\{\rho_n\}_{n=1}^\infty$ denote a sequence of scale factors satisfying $\rho_n\to 0$ and $n\rho_n\to\infty$. Finally, let $G_n$ denote the graph that independently keeps each edge of $A_n$ with probability $\rho_n.$ Then, we have
    \begin{align*}
        \norm{\rho_n^{-1}G_n - \L}_\square\to 0
    \end{align*}
    almost surely.
\end{lemma}
Such random thinning is only one example of thinning a dense convergent sequence. More generally, the point is
that edges must be thinned in a sufficiently uniform way so that, after rescaling by $\rho_n$, the limiting geometry $\L$ is preserved. Finally, given the edge density $\rho_n$, the total number of edges is of order $n^2\rho_n.$ The notion of graph convergence is stable to perturbing a vanishing fraction of edges.

\begin{lemma}\label{perturb}
    Let $\{A_n\}_{n=1}^\infty$ denote a sequence of graphs satisfying
    $$\norm{\rho_n^{-1}A_n - \L}_{\square}\to 0$$
    for a sequence of scale factors $\{\rho_n\}_{n=1}^\infty.$ If $B_n$ differs from $A_n$ on $o(n^2\rho_n)$ many unordered pairs of vertices, 
    $$\norm{\rho_n^{-1}B_n - \L}_{\square}\to 0.$$
\end{lemma}
Compared to thinning, Lemma \ref{perturb} shows that adding or deleting $o(n^2\rho_n)$ edges even in a non-uniform way cannot change the limiting geometry. Bringing all the examples together, we see that convergence is a structural condition on the realized finite populations, shared by deterministic, thinned, and edge-perturbed graph sequences beyond randomly generated examples.

\subsubsection{Units with Latent Positions}

We conclude by providing an interpretation of the examples above in the context of causal inference. Assume that each unit $i$ has a latent position $u_i\in [0,1],$ and further assume the empirical distribution of such latent positions stabilizes as $n$ grows. Then, $\L$ captures the large-scale interaction structure between units with latent positions $u_i,u_j$ via $\L(u_i,u_j).$ This may encode an intensity of interaction (e.g. Section \ref{stochastic_example}) or a more specific deterministic rule (e.g. Section \ref{rule_example}). Moreover, $\ell$ encodes potential outcome functions so that unit $i$'s outcome at exposure $(w,x)$ is given by $\ell(u_i,w,x).$ Finally, $\rho_n$ represents a global constraint on the total number of realized interactions. While literal random thinning is not required, the realized interactions must not be biased toward specific latent positions and instead reflect the underlying interaction structure $\L.$

For example, consider a growing workplace where latent positions encode an employee's role or level. Then, interactions between different employees may depend on their role, captured by $\L.$ Moreover, the realized interactions may depend on time or capacity constraints. As long as these constraints act globally instead of only affecting employees with certain roles, the realized interaction graph may still reflect the underlying structure $\L$. Finally, interactions based on incidental deviations from the role-based pattern $\L$ are allowed, as long as such interactions only constitute a vanishing fraction of the total number of interactions.

The takeaway is that our notion of convergence is plausible when the given finite population reflects an underlying large-scale structure. Random-graph sampling, deterministic discretization, sparse thinning, and edge perturbations are different mechanisms that may generate the given finite population, rather than assumptions required by the design-based causal model.

\section{Central Limit Theorems}\label{CLT_Section}

In this section, we show that our notion of convergence enables design-based asymptotics under global interference. Given a convergent sequence of populations \(\{(G_n,v_n)\}_{n=1}^\infty\), we show that Bernoulli assignment yields \(\sqrt n\)-asymptotic normality of the Horvitz--Thompson and H\'ajek estimators for the average direct effect. We state our formal result under some technical conditions on the scale factors $\{\rho_n\}_{n=1}^\infty$, the kernel $\mathcal L$, and the degree profile of $G_n$. While our notion of convergence captures the large-scale geometry of the exposure graph, the technical conditions characterize our asymptotic regime by regulating the growth of neighborhood sizes. 
\begin{condition}[Scale Factors]\label{Scale}
\normalfont We require
\leavevmode
    \begin{enumerate}[label=(\alph*), ref=\thecondition(\alph*)]
        \item\label{Scalea} $\rho_n=1$ (dense case) or $\rho_n\to 0$ and $n\rho_n\to\infty$ (sparse case),
        \item\label{Scaleb} $\liminf\limits_{n\to\infty} \log\rho_n/\log n > -1.$
    \end{enumerate}
\end{condition}

\begin{condition}[Kernel]\label{Kernel}
\normalfont We require
\leavevmode
    \begin{enumerate}[label=(\alph*), ref=\thecondition(\alph*)]
        \item\label{Kernela} $\mathcal{L}\in L^4([0,1]^2),$
        \item\label{Kernelb} $\int_0^1 \min\{1,\L(u,v)\}\,dv\geq c$ for any $u\in[0,1]$ where $c>0$ is a constant.
    \end{enumerate}
\end{condition}

\begin{condition}[Degree Profile]\label{Deg}
\normalfont We require
\leavevmode
    \begin{enumerate}[label=(\alph*), ref=\thecondition(\alph*)]
        \item\label{Dega} $\sum_{i=1}^n d_i^4 = O(n^5\rho_n^4),$
        \item\label{Degb} $\min\limits_{1\leq i\leq n} d_i =\Omega(n\rho_n),$
    \end{enumerate}
    where $d_i=\deg_{G_n}(i)$.
\end{condition}

Condition \ref{Deg} controls the relative growth of the units' neighborhoods, while Condition \ref{Kernel} imposes analogous regularity on the limiting object $\L$. In sparse regimes, Condition \ref{Scaleb} rules out edge densities that decay too quickly. As shown in Lemma \ref{interpret} of the \hyperref[app]{Appendix}, there exists $\alpha > - 1$ such that $\rho_n\geq n^\alpha$ for all large enough $n.$ In the dense regime where $\rho_n=1,$ the limit object $\L$ is a graphon, so Conditions \ref{Kernela} and \ref{Dega} are automatically satisfied. We now state our main theorem.

\begin{theorem}
\label{Main}
    Let $\{(G_n, v_n)\}_{n=1}^\infty$ be a deterministic sequence of finite populations under anonymous interference. Let $\mathcal{L}$ be a kernel and $\ell: [0,1]\to \F$ be measurable. Finally, let $\{\rho_n\}_{n=1}^\infty$ denote a sequence of scale factors in $(0,1]$. Assign treatments $W_i\overset{\text{i.i.d.}}{\sim} \text{Ber}(\pi)$ where $\pi\in(0,1).$ Under Conditions \ref{Scale}--\ref{Deg}, if 
    $$(G_n,v_n)\lims (\L,\ell),$$
    then we have
    \begin{align*}
        \sqrt{n}(\hat\tau_n^{\HT}(G_n,v_n) - \overline\tau_n(G_n,v_n))&\limd N(0, V^{\HT}_{\L,\ell,\pi}),\\
        \sqrt{n}(\hat\tau_n^{\Haj}(G_n,v_n) - \overline\tau_n(G_n,v_n))&\limd N(0, V^{\Haj}_{\L,\ell,\pi}).
    \end{align*}
    Moreover, the asymptotic variances are given by
    \begin{align}
        V^{\HT}_{\L,\ell,\pi} &= \pi(1-\pi)\E[(\mathcal{R}_i + \mathcal{Q}_i)^2],\qquad
        V^{\Haj}_{\L,\ell,\pi} = \pi(1-\pi)\left(\Var(\mathcal{R}_i+\mathcal{Q}_i) + \left(\E[\mathcal{Q}_i]\right)^2\right)\label{eq:lim_var}
    \end{align}
    where $U_i\overset{\text{i.i.d.}}{\sim} U[0,1]$,
    \begin{align*}
        \mathcal{R}_i &= \frac{\ell(U_i,1,\pi)}{\pi} + \frac{\ell(U_i,0,\pi)}{1-\pi},\qquad
        \mathcal{Q}_i = \E\left[\left.\frac{\mathcal{L}(U_i,U_j)\left(\ell'(U_j,1,\pi) - \ell'(U_j,0,\pi)\right)}{\int_{[0,1]}\L(u,U_j)\,du}\right| U_i\right],
    \end{align*}
    and the derivatives are with respect to the third argument.
\end{theorem}

We emphasize that asymptotic normality in Theorem \ref{Main} only involves randomness from the treatment assignment. The uniform variables $U_i$ describe the limiting variances as in \cite{Wager}. 

\subsection{Proof Sketch of Theorem \ref{Main}}
Our proof of Theorem \ref{Main} is similar for the HT estimator and the H\'{a}jek estimator. Throughout this section, we focus on the HT estimator and write $\hat\tau_n = \hat\tau_n^{\HT}$. The full proof for both the HT and H\'{a}jek estimators is given in the \hyperref[app]{Appendix}. Here, we provide an overview of our proof.

\subsubsection{A Transfer Argument}
We begin by considering an auxiliary sequence of random finite populations. Let $(L_n,\ell_n)$ denote the random finite population drawn from $(\L,\ell),$ based on latent $U_1,\dots, U_n$ as described in \eqref{random_pop} of Section \ref{stochastic_example}. Let the treatments $W_1,\dots, W_n$ be independent of $U_1,\dots, U_n$. Applying Theorem 4 of \cite{Wager} to $\{(L_n,\ell_n)\}_{n=1}^\infty$ gives the following central limit theorem.
\begin{proposition}\label{wager_main}
    Under Conditions \ref{Scale}--\ref{Kernel}, we have 
    \begin{align*}
    \sqrt{n}(\hat\tau_n(L_n,\ell_n) - \overline\tau_n(L_n,\ell_n))\limd N\left(0,V^{\HT}_{\L,\ell,\pi}\right)
\end{align*}
where $V^{\HT}_{\L,\ell,\pi}$ equals the expression given in \eqref{eq:lim_var} of Theorem \ref{Main}.
\end{proposition}

Proposition \ref{wager_main} is a superpopulation CLT, describing the marginal distribution of the centered estimator under repeated draws of both the population $(L_n,\ell_n)$ and the treatment vector $W$. In contrast, our given population $(G_n,v_n)$ is fixed. We bridge these settings by observing that both $(L_n,\ell_n)$ and $(G_n,v_n)$ approximate the same limiting object through the convergences (see Definition \ref{Mine} and Lemma \ref{random_order})
\begin{align}
    (G_n,v_n)\lims (\L,\ell), \qquad (L_n,\ell_n)&\lims (\L,\ell)\text{ a.s.} \label{eq:convergences}
\end{align}
The main technical step is to construct a coupling under which
\begin{align}
    \hat\tau_n(G_n, v_n)-\overline{\tau}_n(G_n,v_n),\qquad\hat\tau_n(L_n,\ell_n) - \overline{\tau}_n(L_n,\ell_n)\label{eq:coupling_quan}
\end{align}
differ by only $o_p(1/\sqrt{n}).$ This allows us to transfer the asymptotic normality of Proposition \ref{wager_main} to our design-based setting by Slutsky's theorem, concluding the proof of Theorem \ref{Main}.

\subsubsection{First-Order Equivalence after Relabeling}
Here, we describe our coupling of \eqref{eq:coupling_quan}. First, we relabel units in $(G_n,v_n)$ and $(L_n,\ell_n)$ to align both populations with the limiting geometry $(\L,\ell).$ Let $\{\varphi_n\}_{n=1}^\infty$ be any deterministic sequence of permutations satisfying the convergence $(G_n,v_n)\lims (\L,\ell)$ as in Definition \ref{Mine}. Next, let $\sigma_n\in S_n$ denote the random permutation that orders $U_1,\dots, U_n$ so that $U_{\sigma_n(1)}\leq\cdots\leq U_{\sigma_n(n)}.$ The relabeled populations are given by $(G_n^\varphi, v_n^\varphi)$ and $(L_n^\sigma,\ell_n^\sigma).$ Assigning the same treatments $W=(W_1,\dots, W_n)$ to both relabeled populations, we obtain our desired coupling
\begin{align}
   \hat\tau_n(G_n^\varphi, v^\varphi_n) - \overline\tau_n(G_n^\varphi, v_n^\varphi),\qquad \hat\tau_n(L_n^\sigma,\ell_n^\sigma) - \overline\tau_n(L_n^\sigma,\ell_n^\sigma).
   \label{eq:actual_coupling}
\end{align}
In particular, each term in \eqref{eq:actual_coupling} is equal in distribution to its corresponding term in \eqref{eq:coupling_quan} as both $\varphi_n,\sigma_n$ are independent of $W.$ The following result is the key consequence of our coupling.
\begin{lemma}\label{coupling_lemma}
    Under the same treatments $W_1,\dots, W_n,$ we have
\begin{align*}
    \hat\tau_n(G_n^\varphi, v^\varphi_n) - \overline\tau_n(G_n^\varphi, v_n^\varphi)  = \hat\tau_n(L_n^\sigma,\ell_n^\sigma) - \overline\tau_n(L_n^\sigma,\ell_n^\sigma) +  o_p(1/\sqrt{n}).
\end{align*}
\end{lemma}
The key idea behind Lemma \ref{coupling_lemma} is that first-order equivalence is driven by the fact that $(G_n^\varphi,v_n^\varphi)$ and $(L_n^\sigma,\ell_n^\sigma)$ have similar finite-population geometry. To show this formally, we use the following linearization based on Lemma 2 of  \cite{Wager}.
\begin{lemma}\label{move}
For any finite population $(A_n,u_n),$ write
    \begin{align*}
        \psi_i(A_n,u_n) = \frac{u_n(i,1,\pi)}{\pi} + \frac{u_n(i,0,\pi)}{1-\pi} + \sum_{j=1}^n \frac{A_n(i,j)}{d_j(A_n)}\big(u_n'(j,1,\pi) - u_n'(j,0,\pi)\big)
    \end{align*}
    where $u_n(i,1,\pi)=u_n(i)(1,\pi)$ and the derivatives are with respect to the third argument. Then, Conditions \ref{Scale}--\ref{Deg} give
    \begin{align*}
        \hat\tau_n(G_n^\varphi, v_n^\varphi) - \overline\tau_n(G_n^\varphi, v_n^\varphi) &= \frac{1}{n}\sum_{i=1}^n \psi_i(G_n^\varphi,v_n^\varphi)(W_i - \pi) + o_p(1/\sqrt{n}),\\
        \hat\tau_n(L_n^\sigma, \ell_n^\sigma) - \overline\tau_n(L_n^\sigma,\ell_n^\sigma) &= \frac{1}{n}\sum_{i=1}^n \psi_i(L_n^\sigma,\ell_n^\sigma)(W_i - \pi) + o_p(1/\sqrt{n}).
    \end{align*}
\end{lemma} 

To conclude, consider the difference
\begin{align}
    \Delta_n = \frac{1}{n}\sum_{i=1}^n \Big(\psi_i(G_n^\varphi,v_n^\varphi) - \psi_i(L_n^\sigma,\ell_n^\sigma)\Big)(W_i-\pi).
    \label{eq:delta}
\end{align}
By applying Lemma \ref{move} and noting $\E[\Delta_n]=0$, our proof of Lemma \ref{coupling_lemma} reduces to showing $\Var[\Delta_n]=o(1/n)$, or equivalently
\begin{align*}
    \frac{1}{n}\sum_{i=1}^n \E\left[\Big(\psi_i(G_n^\varphi,v_n^\varphi) - \psi_i(L_n^\sigma,\ell_n^\sigma)\Big)^2\right] = o(1).
\end{align*}
This is precisely where our notion of convergence enters. The unit-level terms in the $\psi_i$'s are controlled by the convergence of the potential outcome functions. The remaining terms are more delicate as they average derivatives of neighboring units with weights determined by the exposure graph. In the \hyperref[app]{Appendix}, we control such graph-weighted derivative terms via our notion of convergence. We remark that this step is the most technical part of our proof.

\subsection{Random-Graph Models and First-Order Uncertainty}

Based on Theorem \ref{Main}, we give a design-based interpretation of the random-graph asymptotics of \cite{Wager}. Recall that $(L_n,\ell_n)$ denotes the random finite population generated as in \eqref{random_pop}. The following is a conditional version of Proposition \ref{wager_main}.
\begin{corollary}\label{clarify}
    Let $\L$ be a kernel, $\ell:[0,1]\to\F$ be measurable, and assume Conditions \ref{Scale}--\ref{Kernel}. Then, for almost every realization of $\{(L_n,\ell_n)\}_{n=1}^\infty$, we have
    \begin{align*}
        \sqrt{n}(\hat\tau_n^{\HT}(L_n,\ell_n) - \overline\tau_n(L_n,\ell_n))&\limd N(0, V^{\HT}_{\L,\ell,\pi}),\\
        \sqrt{n}(\hat\tau_n^{\Haj}(L_n,\ell_n) - \overline\tau_n(L_n,\ell_n))&\limd N(0, V^{\Haj}_{\L,\ell,\pi}),
    \end{align*}
    conditionally on the realized $\{(L_n,\ell_n)\}_{n=1}^\infty.$ Moreover, the asymptotic variances match \eqref{eq:lim_var} of Theorem \ref{Main}.
\end{corollary}

The proof is immediate from Lemma \ref{random_order} and Theorem \ref{Main}, after verifying that $\{L_n\}_{n=1}^\infty$ satisfies Condition \ref{Deg} almost surely. We give this verification in the \hyperref[app]{Appendix}. By Corollary \ref{clarify}, almost every realization of $\{(L_n,\ell_n)\}_{n=1}^\infty$ yields the same asymptotic variance, and we see that graph randomness does not introduce additional first-order variation. In other words, first-order uncertainty in this setting is driven by treatment assignment rather than graph randomness. Thus, the main regularity provided by the random-graph model is the structural stability captured in Definition \ref{Mine}, rather than randomness from the latent uniform variables $U_1,\dots, U_n.$

This result can be viewed as part of a broader statistical phenomenon, where random design features may not contribute additional first-order uncertainty if their empirical behavior is sufficiently stable. For instance, under well-specified linear regression, conditioning on the realized covariates changes the exact variance of the OLS estimator, but not its asymptotic variance when the empirical Gram matrix stabilizes \citep{buja2019stats}. Similarly, in semiparametric estimation of the ATE under unconfoundedness, knowledge of the propensity score does not improve first-order efficiency \citep{jinyong1998prop}. Corollary \ref{clarify} establishes an analogous point here. Once the exposure graph converges in the appropriate sense and has stable large-scale geometry, its randomness affects only lower-order terms, not the leading asymptotic variance.

To conclude, we compare the limiting variances $V^{\HT}_{\L,\ell,\pi}$ and $V^{\Haj}_{\L,\ell,\pi}$. Recalling $\mathcal{R}_i,\mathcal{Q}_i$ from Theorem \ref{Main}, note that
\begin{align*}
\E[(\mathcal{R}_i+\mathcal{Q}_i)^2] = \Var(\mathcal{R}_i+\mathcal{Q}_i) + \left(\E[\mathcal{R}_i+\mathcal{Q}_i]\right)^2.
\end{align*}
As discussed in \cite{Wager} as well, we see that $V^{\HT}_{\L,\ell,\pi}\geq V^{\Haj}_{\L,\ell,\pi}$ if and only if 
$$ \left(\E[\mathcal{R}_i+\mathcal{Q}_i]\right)^2\geq \left(\E[\mathcal{Q}_i]\right)^2.$$
Under no-interference, we have $\mathcal{Q}_i=0$, recovering the fact that the H\'{a}jek estimator has smaller limiting variance. Under interference, however, we see that it is possible for the HT estimator to have smaller limiting variance, namely when $ \left(\E[\mathcal{R}_i+\mathcal{Q}_i]\right)^2< \left(\E[\mathcal{Q}_i]\right)^2.$

\section{Numerical Illustrations}\label{Numeric}

In this section, we use simulations to evaluate the finite-population behavior of the asymptotic approximation in Theorem \ref{Main}. In particular, we show that the normal approximation can remain accurate even when Conditions \ref{Kernelb} and \ref{Degb} are relaxed. Moreover, we compare the design-based and superpopulation approaches illustrated in Corollary \ref{clarify} and Proposition \ref{wager_main}. Concretely, we fix the limit object $(\L,\ell)$ as
\begin{align*}
    \L(u,v)&=1_{u\neq v}\cdot 1_{u+v>1},\\
    \ell(t,w,x)&= t + (1+4t)w + (2+2t)x + 5x^2 + 4wx.
\end{align*}  
Note that Condition \ref{Kernelb} is not satisfied as $\int_0^1 \L(u,v)\,dv = u.$ With $n=1000$, we consider both a dense regime ($\rho_n=1$) and a sparse regime 
$(\rho_n = n^{-0.3})$ under three different settings. In each case, we use $\text{Ber}(0.5)$ assignment and compare the empirical distribution of the estimator to the Gaussian limit predicted by Theorem \ref{Main}.  

The first setting adopts a design-based perspective, where we generate a convergent finite population by Lemmas \ref{det_point} and \ref{thinning}. In other words, the exposure graph is the half graph $H_n$ where $H_n(i,j)=1_{i\neq j}\cdot 1_{i+j>n},$ and the potential outcome functions are $\ell(i/n)$ for each $i\in[n]$. For the sparse setting, each edge of $H_n$ is independently kept with probability $\rho_n$. Fixing the realized finite population, we generate 10,000 treatment assignments to obtain an empirical plot of the HT and H\'{a}jek estimators. This allows us to evaluate Theorem \ref{Main} under a finite population of fixed size.

In the second and third settings, we draw the finite population from $(\L,\ell)$ according to \eqref{random_pop}. The second setting is design-based as we condition on the realized draw, then generate 10,000 treatment assignments. In contrast, the third setting adopts a superpopulation approach. In each of the 10,000 Monte Carlo iterations, we draw a new finite population along with a new treatment assignment. These two settings allow us to compare the normal approximations in Corollary \ref{clarify} and Proposition \ref{wager_main}.

\begin{figure}[t]
    \centering
    \includegraphics[width=1\textwidth]{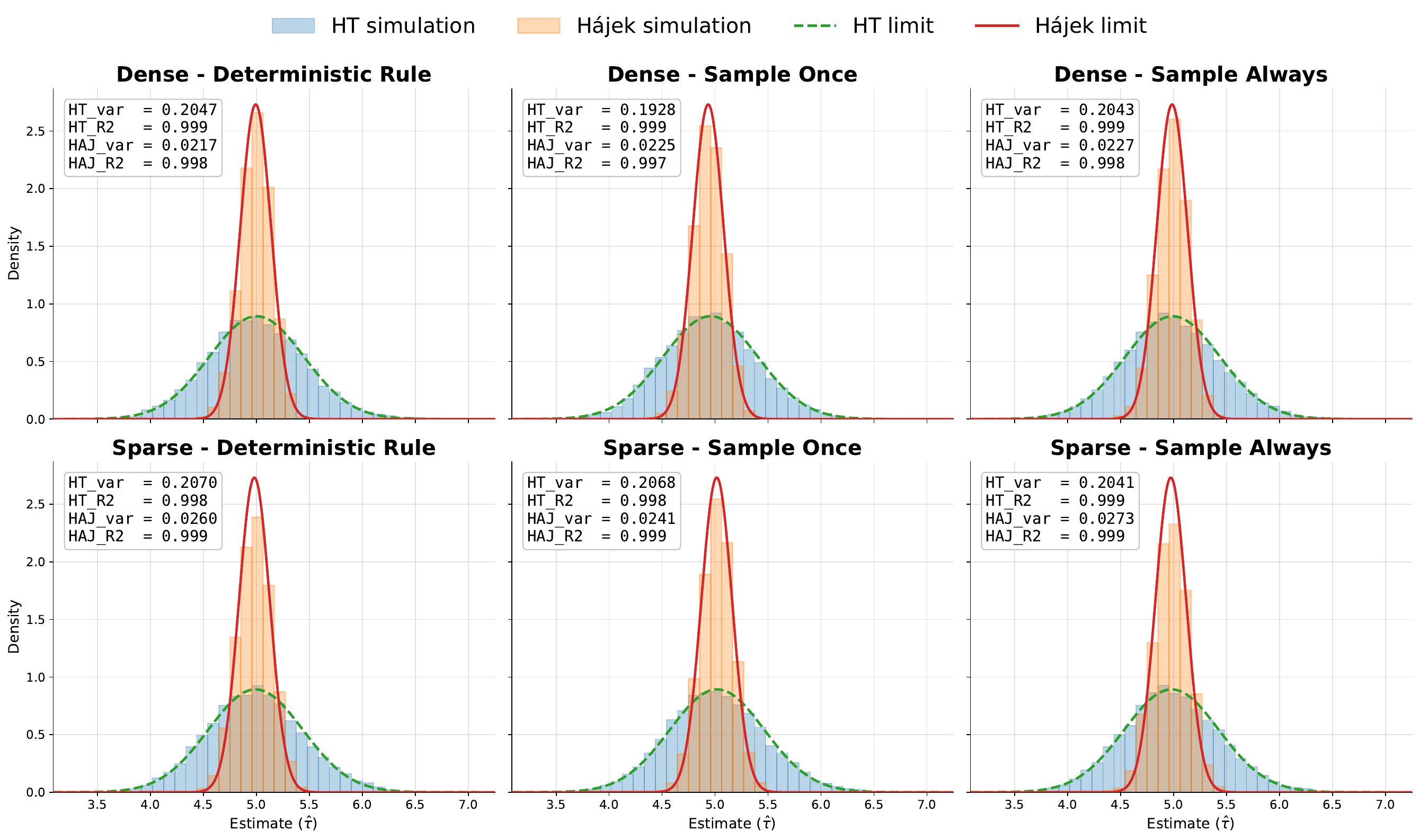}
    \caption{Empirical distributions of the Horvitz--Thompson and H\'{a}jek estimators with $n=1000$. We use 10,000 Monte Carlo iterations. Each plot lists empirical variances. Dividing the true limiting variances by $n$ gives $0.199$ and $0.021$ respectively for the HT and H\'{a}jek estimators.}
    \label{fig:plot}
\end{figure}

We display the results in Figure \ref{fig:plot}. In each plot, we overlay the theoretical limit implied by Theorem \ref{Main}. Namely, we compute the limiting variance and center the corresponding Gaussian curve on the empirical mean. We give a short derivation of the limiting variances in the \hyperref[app]{Appendix}. For each setting, we also include the $R^2$ of the QQ plot against the standard normal distribution to check approximate normality. 

First, the results from the first and second settings show that the normal approximation of Theorem \ref{Main} is reasonable under $n=1000.$ The $R^2$ values are all very close to $1,$ and the empirical variances are close to the limiting variances. In the sparse case, the H\'{a}jek estimator exhibits more visible variance inflation than the HT estimator, which is possibly due to the fact that its limiting variance is much smaller. Next, the shape of the empirical distributions in the second and third settings are similar. This supports our interpretation that first-order uncertainty in estimating the ADE is driven by treatment assignment rather than graph randomness. Finally, as all three settings do not satisfy Conditions \ref{Kernelb} and \ref{Degb}, the results suggest that the asymptotics in Theorem \ref{Main} and Proposition \ref{wager_main} can be robust to degree heterogeneity.

\section{Discussion}

This work provides a notion of convergence for a sequence of finite populations under anonymous interference. As a result, we obtain $\sqrt{n}$-asymptotic normality of the HT and H\'{a}jek estimators for the average direct effect, even on dense, non-random exposure graphs. Moreover, we show that the random-graph model of \cite{Wager} provides an example that generates convergent finite populations. In particular, the first-order uncertainty in their average direct effect estimation is driven by treatment assignment rather than graph randomness. These results show that graph limit theory provides a natural language for expressing asymptotic assumptions in network experiments.

While we focus on asymptotic normality, the interference literature contains several adjacent lines of work. One direction studies optimal designs under interference. For instance, \citet{leung2022spatial,leung2025crosscluster} studies cluster-based designs under spillovers that decay with spatial distance, highlighting a bias-variance tradeoff resulting from the choice of clusters. \citet{viviano2023clustering} study a similar tradeoff in general network interference settings. Another line of work emphasizes randomization-based testing. For instance, \citet{athey2018exact} give a general strategy for testing non-sharp null hypotheses under network interference by passing to an artificial experiment. \citet{basse2024peer} further develop permutation tests for peer effects in random group formation experiments.

We conclude by describing natural directions for future work. First, while we accommodate vanishing edge densities, our asymptotic regime still concerns graphs with growing average degrees. In bounded-degree regimes, the standard notion of graph convergence is Benjamini--Schramm convergence \citep{benjamini2001convg}. Note that asymptotic normality in such settings is often available from dependency-graph CLTs, hence the main issue is not normal approximation itself. Rather, a natural question is whether an appropriate Benjamini--Schramm-type limit can characterize the limiting variance and support variance estimation.

Another natural direction is to extend our results beyond anonymous interference to exposure mappings that depend on neighbor identities. In such settings, the relevant limiting object may need to be richer than an ordinary graphon, for example a decorated graphon. One could then formulate an analogous notion of convergence by requiring that, under a common labeling of units, the enriched network structure and potential outcomes at any common fixed exposure converge jointly. While our proof is based on a transfer argument, one may prove asymptotic normality in such settings by directly utilizing the notion of convergence. This may allow one to weaken some technical conditions as well. For instance, we remark that our Condition \ref{Kernela} is stronger than the $L^2$ condition assumed in \cite{Wager}.

A final direction involves feasible inference. While we establish a central limit theorem for the HT and H\'{a}jek estimators, the limiting variances depend on features of the interference structure and potential outcome functions that are not directly observable. In the random-graph setting, \cite{Fan} develop consistent variance estimators for the (regression-adjusted) Hájek estimator under low-rank assumptions on the underlying graphon. An important open question is to investigate whether their variance estimators can be extended to our design-based setting through our notion of convergence.

\section*{Acknowledgements}
This research was supported by NSF grant SES-2242876. BP acknowledges support from the Stanford Graduate Fellowship. The simulation code was developed using AI assistance. AI tools were also used for limited wording suggestions and sentence-level editing in parts of the manuscript. The authors reviewed all AI-assisted material and are fully responsible for the content of the manuscript.

\bibliographystyle{plainnat} 
\bibliography{references} 

\newpage
\appendix

\section{Finite Population Convergence}\label{app}
\subsection{Convergence under SUTVA}
Here, we recall the setting of Section \ref{SUTVA_sec} and show that convergence according to Definition \ref{fin_convg} implies convergence of the empirical distributions $F_n$.
\begin{lemma}
    Let $\{(a_n,b_n)\}_{n=1}^\infty$ be a sequence of finite populations under SUTVA such that 
    $(a_n,b_n)\to (\ell_0,\ell_1)$ where $\ell_0,\ell_1$ are integrable functions on $[0,1].$ Let $U\sim U[0,1]$ and let $F$ denote the law of $(\ell_0(U),\ell_1(U)).$ Then, we have $F_n\Rightarrow F.$ 
\end{lemma}
\begin{proof}
    Since $(a_n,b_n)\to (\ell_0,\ell_1),$ there is a sequence of permutations $\varphi_n\in S_n$ so that $\norm{a_n^\varphi-\ell_0}_1\to 0$ and $\norm{b_n^\varphi-\ell_1}_1\to 0.$ Note that $F_n$ is the law of $(a_n(U),b_n(U))$ and also $(a_n^\varphi(U),b_n^\varphi(U))$. Moreover,
    \begin{align*}
        \E\norm{(a_n^\varphi(U),b_n^\varphi(U))-(\ell_0(U),\ell_1(U))}_1 &= \E|a_n^\varphi(U)-\ell_0(U)| + \E|b_n^\varphi(U)-\ell_1(U)|\\
        &= \norm{a_n^\varphi-\ell_0}_1 + \norm{b_n^\varphi-\ell_1}_1\\
        &= o(1).
    \end{align*}
    Thus, we see that $(a_n^\varphi(U),b_n^\varphi(U))\to (\ell_0(U),\ell_1(U))$ in $L^1$ and also in probability. In particular, $F_n\Rightarrow F$ as desired.
\end{proof}

\subsection{Convergence under Anonymous Interference}

Here, we prove results in Section \ref{interpret_conv}.

\begin{proof}[Proof of Lemma \ref{random_order}]
    We begin with the first part. First assume that $\rho_n=1$ (dense case) and $\L$ is a graphon. Let $H_n$ denote the weighted adjacency matrix with $(i,j)$th entry $\L(U_i,U_j).$ Given $H_n$, we see that $L_n$ is the random graph generated by independently assigning an edge to $\{i,j\}$ with probability $H_n(i,j).$ By the triangle inequality, we have
        \begin{align*}
            \norm{L_n^\sigma - \L}_\square \leq \norm{L_n^\sigma - H_n^\sigma}_\square + \norm{H_n^\sigma - \L}_\square.
        \end{align*}
    The first term of the RHS goes to zero almost surely by Lemma 4.3 of \citet{graphon} along with the Borel--Cantelli lemma. The second term of the RHS goes to zero almost surely by Theorem 2.14(a) of \citet{key}. Note that \citet{key} use the term ``graphon" to include our notion of kernels. For the sparse case where $\rho_n\to 0$ and $n\rho_n\to\infty,$ the result $\norm{\rho_n^{-1}L_n^\sigma - \L}_\square\to 0$ is directly given as Theorem 2.14(b) of \citet{key}.

    Next, we wish to show that almost surely, $\norm{\ell_n^\sigma(\cdot,w,x) - \ell(\cdot,w,x)}_{1}\to 0$
    for any $(w,x)\in \{0,1\}\times [0,1].$ First, fix any $(w,x).$ Since $\ell(t)\in\F$, we see that the map $t\mapsto \ell(t,w,x)$ is bounded on $[0,1].$ Hence, for any $\varepsilon>0,$ there exists continuous $g$ on $[0,1]$ so that $\norm{g-\ell(\cdot,w,x)}_1 < \varepsilon.$ Now consider $g_n=(g(U_1),\dots, g(U_n))$. Then, we have 
    \begin{align*}
        \norm{\ell_n^\sigma(\cdot,w,x) - \ell(\cdot,w,x)}_{1}\leq \norm{\ell_n^\sigma(\cdot,w,x) - g_n^\sigma}_1 + \norm{g_n^\sigma - g}_1 + \norm{g-\ell(\cdot,w,x)}_1.
    \end{align*}
    The first term on the RHS converges to $\norm{g-\ell(\cdot,w,x)}_1$ by the strong law of large numbers. Moreover, note that $g$ is uniformly continuous on $[0,1].$ Writing the second term on the RHS as a sum of integrals over the intervals of length $1/n$, we see it is bounded above by the modulus of continuity $\omega_g\left(1/n + \max_{i\in [n]}\left|U_{(i)}-i/n\right|\right).$ By the Glivenko--Cantelli theorem, we know that $\max_{i\in [n]}\left|U_{(i)}-i/n\right|\limas 0.$ Hence, we get $\norm{g_n^\sigma - g}_1\limas 0$. Finally taking $\varepsilon\to 0,$ we conclude that $\norm{\ell_n^\sigma(\cdot,w,x) - \ell(\cdot,w,x)}_{1}\to 0$ almost surely. 
    
    By countability, we almost surely have $\norm{\ell_n^\sigma(\cdot,w,x) - \ell(\cdot,w,x)}_{1}\to 0$ for all $(w,x)$ where $x$ is rational. To conclude, fix any $(w,x)$ where $x\in[0,1]\setminus\mathbb{Q}.$ As we assume that elements of $\F$ have  derivatives bounded by $C,$ we see that elements of $\F$ are $C$-Lipschitz in the exposure argument. Thus, we get
    \begin{align*}
        |\ell_n^\sigma(t,w,p)-\ell_n^\sigma(t,w,x)| &\leq C|p-x|\\
        |\ell(t,w,p)-\ell(t,w,x)| &\leq C|p-x|
    \end{align*}
    for any $t\in [0,1]$ and $p\in\Q\cap[0,1]$. Since the rationals are dense in $[0,1],$ we can extend the almost sure convergence to all $x\in [0,1]$. This concludes the proof.
    \end{proof}

\begin{proof}[Proof of Lemma \ref{det_point}]
   We begin with the graphon. Since $\L$ is Riemann integrable, we see that it is continuous almost everywhere on $[0,1]^2.$ Next, let $(u,v)\in [0,1]^2$ be any point such that $\L$ is continuous at $(u,v)$ and both $u,v$ are irrational. Moreover, define $i_n,j_n\in [n]$ such that 
   \begin{align*}
       \frac{i_n-1}{n} < u < \frac{i_n}{n},\qquad \frac{j_n-1}{n} < v < \frac{j_n}{n}.
   \end{align*}
   Rewriting the inequalities as 
   \begin{align*}
       u < \frac{i_n}{n} < u + \frac{1}{n},\qquad v < \frac{j_n}{n} < v + \frac{1}{n},
   \end{align*}
   we see that $i_n/n\to u$ and $j_n/n\to v$ as $n\to\infty.$ To conclude, we have
   \begin{align*}
       G_n(u,v) = \L\left(\frac{i_n}{n},\frac{j_n}{n}\right)\to \L(u,v)
   \end{align*}
   for almost every $(u,v)\in [0,1]^2.$ By the dominated convergence theorem, we have   
   \begin{align*}
       \norm{G_n-\L}_1\to 0
   \end{align*}
   as desired. The proof for each $\ell(\cdot,w,x)$ is identical to the above argument.
\end{proof}

\begin{proof}[Proof of Lemma \ref{thinning}]
    By the triangle inequality, we can write 
    \begin{align*}
         \norm{\rho_n^{-1}G_n - \L}_\square\leq  \norm{\rho_n^{-1}G_n - A_n}_\square +  \norm{A_n - \L}_\square.
    \end{align*}
    We have $\norm{\rho_n^{-1}G_n - A_n}_\square\to 0$ by Lemma 7.3 of \cite{key}. Since $\norm{A_n-\L}_\square\to 0$ by assumption, we conclude the proof.
\end{proof}

\begin{proof}[Proof of Lemma \ref{perturb}]
    Beginning with the triangle inequality, we have
    \begin{align*}
        \norm{\rho_n^{-1}B_n - \L}_{\square}&\leq \norm{\rho_n^{-1}B_n - \rho_n^{-1}A_n}_{\square} + \norm{\rho_n^{-1}A_n - \L}_{\square}\\
        &\leq \rho_n^{-1}\norm{A_n-B_n}_\square  + o(1)\\
        &\leq \rho_n^{-1}\norm{A_n-B_n}_1 + o(1).
    \end{align*}
    Noting $\norm{A_n-B_n}_1$ equals $n^{-2}$ times twice the number of edges on which $A_n$ and $B_n$ differ, we conclude that $\norm{A_n-B_n}_1=n^{-2}\cdot o(n^2\rho_n) = o(\rho_n)$ and conclude the proof.
\end{proof}

\section{Technical Lemmas for Degrees}
    Here, we give useful lemmas regarding the degrees of $G_n$ and $L_n$ that will be used in the sections below. Given a graph $A_n,$ we will write $d_i(A)$ instead of $d_i(A_n)$ for the sake of brevity. First, we give an interpretation of Conditions \ref{Scale} and \ref{Deg}.
    \begin{lemma}\label{interpret}
        Recall that $d_i(G)= \deg_{G_n}(i)\lor 1.$ Then, under Conditions \ref{Scale} and \ref{Deg}, there are constants $\alpha>-1$ and $c', C'>0$ so that for all large enough $n,$ we have
        \begin{enumerate}[label=(\alph*)]
            \item $\rho_n\geq n^\alpha,$
            \item $\min\limits_{1\leq i\leq n}d_i(G)\geq c'\cdot  n\rho_n,$
            \item $\sum_{i=1}^n d_i(G)^4\leq C'n^5\rho_n^4,$
            \item $\sum_{i=1}^n d_i(G)^2\leq \sqrt{C'}n^3\rho_n^2,$
            \item $\sum_{i=1}^n d_i(G) \leq (C')^{1/4}n^2\rho_n.$
        \end{enumerate}
    \end{lemma}
    \begin{proof}
        These are simple interpretations of the conditions in Theorem \ref{Main} since for large enough $n$, we have $d_i(G)=\deg_{G_n}(i)$. The final three inequalities hold by repeatedly applying Cauchy--Schwarz.
    \end{proof}

    Next, we list some consequences of Condition \ref{Kernel}. The first result regards the minimum degree of $L_n.$
    \begin{lemma}\label{rare}
        Recall that $d_i(L)=\deg_{L_n}(i)\lor 1.$ Let $E_n$ denote the event $$E_n=\left\{\min\limits_{1\leq i\leq n}d_i(L)\geq c/2\cdot n\rho_n\right\}$$ where $c$ is the constant described in Condition \ref{Kernel}. Then, under Condition \ref{Kernel}, $$\P[E_n']\leq n\exp\left(-Kn\rho_n\right)$$ where $K>0$ is a constant. 
    \end{lemma}
    \begin{proof}
    Note that $d_i(L)\geq \deg_{L_n}(i).$ The result follows from Lemma 15 of \citet{Wager}.
    \end{proof}
    The second result involves empirical moments of the degrees.
    \begin{lemma}\label{exp}
        Under Conditions \ref{Scale}--\ref{Kernel}, there are constants $D_k>0$ (depending on $\L$) where $k\in \{1,2,3,4\}$ so that
        \begin{align*}
            \E\left[\sum_{i=1}^n d_i(L)^k\right]&\leq D_k\cdot n^{k+1}\rho_n^k.
        \end{align*}
    \end{lemma}
    \begin{proof}
        By Lemma 15 of \citet{Wager}, we know that there are constants $V_1,\dots, V_4>0$ (depending on $\L$) so that
        \begin{align*}
            \E[\deg_{L_n}(i)]^k\leq V_k (n\rho_n)^k.
        \end{align*}
        Next, we have $d_i(L)\leq \deg_{L_n}(i)+1.$ Hence, we have
        \begin{align*}
            \E\left[d_i(L)^k\right]\leq \E\left[(\deg_{L_n}(i)+1)^k\right]\leq \sum_{\ell=0}^k \binom{k}{\ell}V_\ell\cdot (n\rho_n)^\ell
        \end{align*}
        where $V_0=1.$ This shows that $\E[d_i(L)^k]=O(n^k\rho_n^k)$ for each $k\in\{1,2,3,4\}$ (by the same Big-O constant over $i\in [n]$). Summing over $i\in [n]$ gives our desired result.
    \end{proof}

Finally, we show that the degrees of $L_n$ also satisfy Condition \ref{Deg}.

\begin{lemma}\label{concentration_deg}
    Under Conditions \ref{Scale}--\ref{Kernel}, the sequence $\{L_n\}_{n=1}^\infty$ almost surely satisfies Condition \ref{Deg}.
\end{lemma}
\begin{proof}
Let $d_i$ denote the degree of $i$ in $L_n.$ The key is to note that
    $$d_i\mid U_i\sim \text{Bin}(n-1,g_n(U_i))$$
    where $g_n(u)=\int_0^1\min\{1,\rho_n\L(u,v)\}\,dv$. 
    By Condition \ref{Kernelb}, we have
    \begin{align*}
        g_n(u)\geq \int_0^1 \rho_n\min\{1,\L(u,v)\}\,dv\geq c\rho_n.
    \end{align*}
    Similarly, we have
    \begin{align*}
        g_n(u)\leq \rho_n\int_0^1\L(u,v)\,dv = \rho_n\lambda(u)
    \end{align*}
    where $\lambda(u)=\int_0^1\L(u,v)\,dv.$
    
    Letting $\mu_i = (n-1)g_n(U_i)\geq (n-1)c\rho_n,$ a Chernoff bound gives
    \begin{align*}
        \P(|d_i-\mu_i| > \mu_i/2\mid U_i)\leq 2\exp(-C\mu_i)\leq 2\exp(-C'n\rho_n)
    \end{align*}
    for some constants $C$ and $C'$. Taking unconditional probabilities and applying the union bound, we see that
    \begin{align*}
        \P(\exists i: |d_i-\mu_i|>\mu_i/2)\leq 2n\exp(-C'n\rho_n)\leq \frac{2n}{e^{C'n^{1+\alpha}}}
    \end{align*}
    for some $\alpha>-1$ by Condition \ref{Scaleb}. As the rightmost side is summable over $n,$ by Borel--Cantelli, we almost surely have
    \begin{align*}
       \frac{\mu_i}{2}\leq d_i\leq \frac{3}{2}\mu_i  \qquad \forall i\in [n]
    \end{align*}
    for all large enough $n.$ Since $\mu_i\geq (n-1)c\rho_n,$ we see that Condition \ref{Degb} is satisfied almost surely.

    To conclude, note that
    \begin{align*}
        \sum_{i=1}^n d_i^4&\leq \left(\frac{3}{2}\right)^4\sum_{i=1}^n(n-1)^4g_n(U_i)^4\leq C''n^5\rho_n^4\cdot \frac{1}{n}\sum_{i=1}^n \lambda(U_i)^4
    \end{align*}
    for some constant $C''.$ Since 
    \begin{align*}
        \E[\lambda(U_1)^4] = \int_0^1\lambda(u)^4\,du \leq \int_0^1\int_0^1\L(u,v)^4\,dv\,du < \infty
    \end{align*}
    by Condition \ref{Kernela}, the strong law of large numbers gives 
    \begin{align*}
        \frac{1}{n}\sum_{i=1}^n \lambda(U_i)^4 \limas \E[\lambda(U_1)^4].
    \end{align*}
    In particular, we see that Condition \ref{Dega} is also almost surely satisfied and conclude the proof.
\end{proof}

\section{Linearization of the HT and H\'{a}jek Estimators}\label{linearization_app}

In this section, we show how to linearize the HT and H\'{a}jek estimators. We first prove Lemma \ref{move}.

\begin{proof}[Proof of Lemma \ref{move}]
We write $\widehat\tau_n = \widehat\tau_n^{\HT}.$ First, Lemma 2 of \citet{Wager} (with different notation) states that conditional on any finite population $(A_n,u_n)$, we have
    \begin{align*}
    \hat\tau_n(A_n,u_n)-\overline{\tau}_n(A_n,u_n) = 
        \frac{1}{n}\sum_{i=1}^n \psi_i(A_n,u_n)(W_i-\pi) + O_p(R(A_n))
    \end{align*}
    where 
    \begin{align*}
        R(A_n) &= C\left(\frac{1}{\sqrt{n \cdot \delta_n}} + \frac{\sqrt{\sum_{i,j}\gamma_{i,j}}}{n\cdot  \delta_n^{3/2}}\right),\\
        \delta_n &= \min\limits_{1\leq i\leq n}\deg_{A_n}(i),\\
        \gamma_{i,j}&= |\{k\in [n]\setminus\{i,j\}: k\sim i, k\sim j\}|.
    \end{align*}
    In particular, this result holds unconditionally as well. Applying this result to each of the coupled graphs $L_n^\sigma$ and $G_n^\varphi,$ it remains to show that $R(L_n^\sigma)=o_p(1/\sqrt{n})$ and $R(G_n^\varphi)=o_p(1/\sqrt{n})$ under Conditions \ref{Scale}--\ref{Deg}. 
    
    We claim that the quantity $R(A_n)$ is invariant under relabeling of vertices in $A_n$. To see this, for any $\eta\in S_n,$ we have
    \begin{align*}
        \sum_{i=1}^n\sum_{j=1}^n|\{k\in [n]\setminus\{i,j\}:k\sim_{A^\eta} i, k\sim_{A^\eta} j\}| &= \sum_{i=1}^n\sum_{j=1}^n|\{k\in [n]\setminus\{i,j\}:\eta(k)\sim_{A} \eta(i), \eta(k)\sim_{A} \eta(j)\}|\\
        &= \sum_{i=1}^n\sum_{j=1}^n|\{w\in [n]\setminus\{\eta(i),\eta(j)\}:w\sim_{A} \eta(i), w\sim_{A} \eta(j)\}|\\
        &= \sum_{u=1}^n\sum_{v=1}^n|\{w\in [n]\setminus\{u,v\}:w\sim_{A} u, w\sim_{A} v\}|
    \end{align*}
    along with 
    \begin{align*}
        \min_{1\leq i\leq n}\deg_{A_n}(i) = \min_{1\leq i\leq n}\deg_{A_n^\eta}(i).
    \end{align*}
    This shows that $\delta_n$ and $\sum_{i,j}\gamma_{i,j}$ are indeed invariant under relabeling. Hence, we may consider $R(L_n)$ and $R(G_n)$ instead of $R(L_n^\sigma)$ and $R(G_n^\varphi).$ Finally, for any graph $A_n,$ note that
    \begin{align}
        \sum_{i,j}\gamma_{i,j} &\leq \sum_{i,j}\sum_{k=1}^n 1_{k\sim i}1_{k\sim j}= \sum_{k=1}^n \sum_{i,j}1_{k\sim i}1_{k\sim j} = \sum_{k=1}^n \deg_{A_n}(k)^2\leq \sum_{k=1}^n d_k(A)^2.\label{eq:useful}
    \end{align}

We first show that $R(L_n)=o_p(1/\sqrt{n}).$ Fix any $\varepsilon>0$ and recall the event $E_n$ from Lemma \ref{rare}. Then, we get
\begin{align*}
    \P\left[R(L_n) > \frac{\varepsilon}{\sqrt{n}}\right]&\leq \P\left[R(L_n) > \frac{\varepsilon}{\sqrt{n}},E_n\right] + \P[E_n']\\
    &\leq \P\left[R(L_n)1_{E_n} > \frac{\varepsilon}{\sqrt{n}}\right] + n\exp(-Kn\rho_n).
\end{align*}
By Markov's inequality along with Lemma \ref{interpret}, we further get
\begin{align*}
    \P\left[R(L_n) > \frac{\varepsilon}{\sqrt{n}}\right]
    &\leq \frac{\sqrt{n}}{\varepsilon}\cdot\E\left[R(L_n)1_{E_n}\right] + \frac{n}{e^{Kn^{1+\alpha}}}
\end{align*}
for large enough $n.$ Since $1+\alpha>0,$ the second term above is $o(1).$ To conclude, using \eqref{eq:useful} along with Lemma \ref{exp}, we get
\begin{align*}
    \E[R(L_n)1_{E_n}]&\leq C\left(\frac{1}{\sqrt{c/2\cdot n^2\rho_n}} + \frac{\E\sqrt{\sum_{k=1}^n d_k(L)^2}}{n\cdot (c/2 \cdot n \rho_n)^{3/2}}\right)\\
    &\leq C\left(\frac{1}{\sqrt{c/2\cdot n^2\rho_n}} + \frac{\sqrt{D_2\cdot n^3\rho_n^2}}{n\cdot (c/2 \cdot n \rho_n)^{3/2}}\right)
\end{align*}
where the second inequality is by Jensen's inequality. Simplifying the expression above, since $n\rho_n\to\infty,$ we see that $\E[R(L_n)1_{E_n}]=o(1/\sqrt{n}).$ Therefore, we conclude that $\P[R(L_n)>\varepsilon/\sqrt{n}]=o(1)$ as desired.

Next, we show that $R(G_n)=o(1/\sqrt{n}).$ This follows directly from \eqref{eq:useful} along with Lemma \ref{interpret}. Indeed, for large enough $n,$ we get
\begin{align*}
    R(G_n) &\leq C\left(\frac{1}{\sqrt{c'\cdot n^2\rho_n}} + \frac{\sqrt{\sum_{k=1}^n d_k(G)^2}}{n\cdot (c' \cdot n \rho_n)^{3/2}}\right)\\
    &\leq C\left(\frac{1}{\sqrt{c'\cdot n^2\rho_n}} + \frac{\sqrt{\sqrt{C'}n^3\rho_n^2}}{n\cdot (c' \cdot n \rho_n)^{3/2}}\right) = o(1/\sqrt{n}),
\end{align*}
again since $n\rho_n\to\infty.$ This concludes the proof.
\end{proof}

Next, we prove an analogous linearization for the H\'{a}jek estimator.

\begin{lemma}\label{move_Haj}
For any finite population $(A_n,u_n),$ write
    \begin{align*}
        &\nu_i(A_n,u_n) \\
        &= \frac{u_n(i,1,\pi)}{\pi} + \frac{u_n(i,0,\pi)}{1-\pi} - \frac{1}{n}\sum_{j=1}^n \left(\frac{u_n(j,1,\pi)}{\pi} + \frac{u_n(j,0,\pi)}{1-\pi}\right) + \sum_{j=1}^n \frac{A_n(i,j)}{d_j(A_n)}\big(u_n'(j,1,\pi) - u_n'(j,0,\pi)\big)
    \end{align*}
    where $u_n(i,1,\pi)=u_n(i)(1,\pi)$ and the derivatives are with respect to the third argument. Then, Conditions \ref{Scale}--\ref{Deg} give
    \begin{align*}
        \hat\tau_n^{\Haj}(G_n^\varphi, v_n^\varphi) - \overline\tau_n(G_n^\varphi, v_n^\varphi) &= \frac{1}{n}\sum_{i=1}^n \nu_i(G_n^\varphi,v_n^\varphi)(W_i - \pi) + o_p(1/\sqrt{n}),\\
        \hat\tau_n^{\Haj}(L_n^\sigma, \ell_n^\sigma) - \overline\tau_n(L_n^\sigma,\ell_n^\sigma) &= \frac{1}{n}\sum_{i=1}^n \nu_i(L_n^\sigma,\ell_n^\sigma)(W_i - \pi) + o_p(1/\sqrt{n}).
    \end{align*}
\end{lemma}
\begin{proof}
    By the proof of Lemma 2 of \citet{Wager} (given in their Appendix), conditional on any finite population $(A_n,u_n),$ we have
    \begin{align*}
    \hat\tau_n^{\Haj}(A_n,u_n)-\overline{\tau}_n(A_n,u_n) = 
        \frac{1}{n}\sum_{i=1}^n \nu_i(A_n,u_n)(W_i-\pi) + O_p(R(A_n))
    \end{align*}
    where $R(A_n)$ is defined in the proof of Lemma \ref{move}. By proceeding identically as in the proof of Lemma \ref{move}, we show that $R(L_n)=o_p(1/\sqrt{n})$ and $R(G_n)=o(1/\sqrt{n}),$ concluding the proof.
\end{proof}

\section{Coupling Estimators for $(G_n^\varphi,v_n^\varphi)$ and $(L_n^\sigma,\ell_n^\sigma)$}\label{decompose_app}

In this section, we prove Lemma \ref{coupling_lemma} for the HT estimator and an analogous result for the H\'{a}jek estimator. We begin with Lemma \ref{coupling_lemma}. Recall from \eqref{eq:delta} that 
\begin{align*}
    \Delta_n = \frac{1}{n}\sum_{i=1}^n \Big(\psi_i(G_n^\varphi,v_n^\varphi) - \psi_i(L_n^\sigma,\ell_n^\sigma)\Big)(W_i-\pi).
\end{align*}
By Lemma \ref{move}, the proof of Lemma \ref{coupling_lemma} reduces to showing that $\Delta_n=o_p(1/\sqrt{n}).$ For this purpose, we write $\Delta_n = D_1 + D_2$ where 
\begin{align*}
        D_1&=\frac{1}{n}\sum_{i=1}^n \left(\frac{v_n^\varphi(i,1,\pi)-\ell_n^\sigma(i,1,\pi)}{\pi} + \frac{v_n^\varphi(i,0,\pi)-\ell_n^\sigma(i,0,\pi)}{1-\pi}\right)(W_i-\pi),\\
        D_2 &= \frac{1}{n}\sum_{i=1}^n \left(\sum_{j=1}^n \frac{G_{ij}^\varphi}{d_j(G^\varphi)}((v_{n}^\varphi)'(j,1,\pi) - (v_{n}^\varphi)'(j,0,\pi)) - \frac{L_{ij}^\sigma}{d_{j}(L^\sigma)}((\ell_{n}^\sigma)'(j,1,\pi) - (\ell_{n}^\sigma)'(j,0,\pi))\right)(W_i-\pi).
    \end{align*}
Since $\E[D_1]=\E[D_2]=0$, we will show that \begin{align}
    \Var(D_1) &= o(1/n),\label{eq:D1}\\
    \Var(D_2) &= o(1/n),\label{eq:D2}
\end{align}
thus implying $\Delta_n=o_p(1/\sqrt{n}).$

In order to control $\Var(D_1)$ and $\Var(D_2)$, we use the fact that $(G_n^\varphi,v_n^\varphi)$ and $(L_n^\sigma,\ell_n^\sigma)$ are asymptotically similar. Combining Definition \ref{Mine} with Lemma \ref{random_order}, for any $(w,x)\in\{0,1\}\times(0,1),$ we almost surely have
\begin{align}
&\rho_n^{-1}\norm{G_n^\varphi - L_n^\sigma}_\square\to 0,\label{eq:graph_conv}\\
&\norm{v_n^\varphi(\cdot,w,x) -\ell_n^\sigma(\cdot,w,x)}_{1}\to 0\label{eq:pot_conv},\\
&\norm{(v_n^\varphi)'(\cdot,w,x) - (\ell_n^\sigma)'(\cdot,w,x)}_{1}\to 0. \label{eq:nice2}
\end{align}
Note that \eqref{eq:graph_conv} and \eqref{eq:pot_conv} follow from the triangle inequality. We now prove \eqref{eq:nice2}.
\begin{proof}[Proof of \eqref{eq:nice2}]
        For any $h\in(0, 1 - x),$ the triangle inequality gives
        \begin{align}
            &\left|(v_n^\varphi)'(t,w,x) - (\ell_n^\sigma)'(t,w,x)\right|\nonumber \\
            &\leq  T_1 + \left|\frac{v_n^\varphi(t,w,x+h)-v_n^\varphi(t,w,x)}{h}-\frac{\ell_n^\sigma(t,w,x+h)-\ell_n^\sigma(t,w,x)}{h}\right| + T_2 \nonumber \\
            &=  T_1 + \left|\frac{v_n^\varphi(t,w,x+h)-\ell_n^\sigma(t,w,x+h)}{h}-\frac{v_n^\varphi(t,w,x)-\ell_n^\sigma(t,w,x)}{h}\right| + T_2 \label{eq:long}
        \end{align}
        where 
        \begin{align*}
            T_1 &= \left|(v_n^\varphi)'(t,w,x)- \frac{v_n^\varphi(t,w,x+h)-v_n^\varphi(t,w,x)}{h}\right|,\\
            T_2 &= \left|(\ell_n^\sigma)'(t,w,x) - \frac{\ell_n^\sigma(t,w,x+h)-\ell_n^\sigma(t,w,x)}{h}\right|.
        \end{align*}
        Since we assume that elements of $\F$ have second derivatives bounded by $C,$ we know that derivatives of elements of $\F$ are $C$-Lipschitz. Thus, by the mean-value theorem, we see that $T_1,T_2\leq Ch.$ Taking the integral over $t\in[0,1]$ in \eqref{eq:long}, then the limsup as $n\to\infty$, and finally $h\to 0$, we get our desired result. Of course, we use the almost sure convergence stated in equation \eqref{eq:pot_conv}.
\end{proof} 

The key input for controlling $\Var(D_1)$ is \eqref{eq:pot_conv}, while the key inputs for $\Var(D_2)$ are \eqref{eq:graph_conv} and \eqref{eq:nice2}. Here, we show that $\Var(D_1)=o(1/n)$ while the proof of $\Var(D_2)=o(1/n)$ is given in the next section. 
\begin{proof}[Proof of \eqref{eq:D1}]
Recall that 
    \begin{align*}
        D_1&=\frac{1}{n}\sum_{i=1}^n \left(\frac{v_n^\varphi(i,1,\pi)-\ell_n^\sigma(i,1,\pi)}{\pi} + \frac{v_n^\varphi(i,0,\pi)-\ell_n^\sigma(i,0,\pi)}{1-\pi}\right)(W_i-\pi).
    \end{align*}
    Conditioning on $(v_n^\varphi,\ell_n^\sigma)$, which is independent of the $W_i$'s, the law of total variance gives
    \begin{align*}
        n\Var(D_1) = \frac{\pi(1-\pi)}{n}\sum_{i=1}^n \E\left[\left(\frac{v_n^\varphi(i,1,\pi)-\ell_n^\sigma(i,1,\pi)}{\pi} + \frac{v_n^\varphi(i,0,\pi)-\ell_n^\sigma(i,0,\pi)}{1-\pi}\right)^2\right].
    \end{align*}
    Fixing $w\in \{0,1\}$ and $\varepsilon > 0,$ define the discrepancy set
    $$S_w = \{i\in [n]: |v_{n}^\varphi(i,w,\pi) - \ell_{n}^\sigma(i,w,\pi)| > \varepsilon\}.$$
    Then, we see that
    \begin{align*}
       \norm{v_n^\varphi(\cdot, w,\pi)-\ell_n^\sigma(\cdot, w,\pi)}_1 \geq \varepsilon\cdot \frac{|S_w|}{n}
    \end{align*}
    which implies $|S_w| = o_p(n)$ by \eqref{eq:pot_conv}. Let $S=S_0\cup S_1$. On $S'=[n]\setminus S,$ both discrepancies are at most $\varepsilon,$ while on $S$ we can use uniform boundedness of the potential outcomes assumed in \eqref{eq:tech_potential}. Therefore, we get
    \begin{align*}
        \frac{1}{n}\sum_{i=1}^n \left(\frac{v_n^\varphi(i,1,\pi)-\ell_n^\sigma(i,1,\pi)}{\pi} + \frac{v_n^\varphi(i,0,\pi)-\ell_n^\sigma(i,0,\pi)}{1-\pi}\right)^2 &\leq \frac{|S'|}{n}\cdot \left(\frac{\varepsilon}{\pi} + \frac{\varepsilon}{1-\pi}\right)^2 + 
        \frac{|S|}{n}
        \cdot \left(\frac{2C}{\pi} + \frac{2C}{1-\pi}\right)^2\\
        &\leq \left(\frac{\varepsilon}{\pi} + \frac{\varepsilon}{1-\pi}\right)^2 + o_p(1)
    \end{align*}
    where the $o_p(1)$ term is uniformly bounded over $n$. Taking expectations on both sides, then the limsup as $n\to\infty,$ and finally letting $\varepsilon\to 0,$ we conclude that $n\Var(D_1)=o(1)$ as desired.
\end{proof}
\subsection{Proof of $\Var(D_2)=o(1/n)$}\label{d2_app}
Here, we show that $\Var(D_2)=o(1/n)$ and conclude the proof of Lemma \ref{coupling_lemma}. We work with the coupled graphs and suppress $(\varphi_n,\sigma_n)$ from the notation: When we write $(G_n,v_n)$ and $(L_n,\ell_n)$ below, we mean
$(G_n^{\varphi},v_n^{\varphi})$ and $(L_n^{\sigma},\ell_n^{\sigma})$ respectively. Moreover, we write $d_j(G)=d_j(G^\varphi_n)$ and $d_j(L)=d_j(L^\sigma_n)$ along with $G_{ij}=G_n^\varphi(i,j)$ and $L_{ij}=L_n^\sigma(i,j)$. 

\begin{proof}[Proof of \eqref{eq:D2}]
    
Take $n$ large enough so that the results of Lemma \ref{interpret} hold. Further recall that
    \begin{align*}
        D_2 = \frac{1}{n}\sum_{i=1}^n \left(\sum_{j=1}^n \frac{G_{ij}}{d_j(G)}x_j - \frac{L_{ij}}{d_j(L)}y_j\right)(W_i-\pi) 
    \end{align*}
    where $x_j = v_{n}'(j,1,\pi) - v_{n}'(j,0,\pi)$ and $y_j = \ell'_{n}(j,1,\pi) - \ell'_{n}(j,0,\pi).$ By \eqref{eq:nice2}, we know that $\norm{\mathbf{x}-\mathbf{y}}_1\to 0$ almost surely (where we embed the vectors $\mathbf{x}=(x_1,\dots, x_n)$ and $\mathbf{y}=(y_1,\dots, y_n)$ as a function on $[0,1]$). By the law of total variance, it suffices to show that
    \begin{align}
        n \Var(D_2) &= \frac{\pi(1-\pi)}{n}\sum_{i=1}^n \E\left[\left(\sum_{j=1}^n \frac{G_{ij}}{d_j(G)}x_j - \frac{L_{ij}}{d_j(L)}y_j\right)^2\right] = o(1).
        \label{eq: Variance2}
    \end{align}

    First, since $|x_j|,|y_j|\leq 2C$ and $d_j(G),d_j(L)\geq 1,$ note that 
    \begin{align*}
        \frac{1}{n}\sum_{i=1}^n \left(\sum_{j=1}^n \frac{G_{ij}}{d_j(G)}x_j - \frac{L_{ij}}{d_j(L)}y_j\right)^2 &\leq \frac{1}{n}\sum_{i=1}^n \left(\sum_{j=1}^n 4C\right)^2 = 16C^2 n^2.
    \end{align*}
    Recalling the event $E_n$ from Lemma \ref{rare}, we have
    \begin{align*}
        \frac{\pi(1-\pi)}{n}\sum_{i=1}^n \E\left[\left(\sum_{j=1}^n \frac{G_{ij}}{d_j(G)}x_j - \frac{L_{ij}}{d_j(L)}y_j\right)^2\cdot 1_{E_n'}\right]&\leq \pi(1-\pi)16C^2n^2\P[E'_n].
    \end{align*}
    Further note that $$n^2\P[E'_n]\leq \frac{n^3}{e^{Kn\rho_n}}\leq\frac{n^3}{e^{Kn^{1+\alpha}}} = o(1)$$
    by Lemmas \ref{interpret} and \ref{rare}. Thus, for the remaining calculations, we will assume that the event $E_n$ holds.
    
    In particular, by the identity $(a_1 + a_2 + a_3)^2 \leq 3(a_1^2 + a_2^2 + a_3^2),$ it suffices to show that
    \begin{align*}
        \pi(1-\pi)(M_1+M_2+M_3) = o(1)
    \end{align*}
    where we write
    \begin{align*}
        M_1 &= \frac{3}{n}\sum_{i=1}^n \E\left[\left(\sum_{j=1}^n \left(\frac{G_{ij}}{d_j(G)} - \frac{G_{ij}}{d_j(L)}\right)x_j\right)^2\cdot 1_{E_n}\right]\\
        M_2 &= \frac{3}{n}\sum_{i=1}^n \E\left[\left(\sum_{j=1}^n \left(\frac{G_{ij}}{d_j(L)} - \frac{L_{ij}}{d_j(L)}\right)x_j\right)^2\cdot 1_{E_n}\right]\\
        M_3 &= \frac{3}{n}\sum_{i=1}^n \E\left[\left(\sum_{j=1}^n \frac{L_{ij}}{d_j(L)}(x_j-y_j)\right)^2\cdot 1_{E_n}\right].
    \end{align*}
    
    We will show that $M_1,M_2,M_3$ are each $o(1)$. In each subsection, $S$ will denote a discrepancy set of vertices.
    \subsubsection{$M_1=o(1)$}
    Assume that the event 
    $E_n=\left\{\min\limits_{1\leq i\leq n}d_i(L)\geq c/2\cdot n\rho_n\right\}$
    holds. Fix any $\varepsilon > 0$ and let 
    \begin{align*}
        S_i&=\{j\in [n]: G_{ij}=1, |d_j(G) - d_j(L)|>\varepsilon \cdot n\rho_n\},\\
        S^+&= \{j\in [n]: d_j(G) - d_j(L)>\varepsilon \cdot n\rho_n\},\\
        S^-&= \{j\in [n]: d_j(G) - d_j(L)<-\varepsilon \cdot n\rho_n\}
    \end{align*}
    along with $S=S^+\cup S^-.$ Next, by Lemma \ref{interpret},
    $$\frac{|d_j(G)-d_j(L)|}{d_j(G)d_j(L)}\leq \frac{1}{d_j(G)}+\frac{1}{d_j(L)}\leq \frac{1}{c'\cdot n\rho_n} + \frac{1}{c/2\cdot n\rho_n}\leq \frac{c''}{n\rho_n}$$
    for some constant $c''>0.$ Fixing $i\in [n],$ consider vertices in $S_i$ and $S_i'$ separately along with the inequality $(a+b)^2\leq 2(a^2+b^2)$  to get
    \begin{align*}
        \left(\sum_{j=1}^n \left(\frac{G_{ij}}{d_j(G)} - \frac{G_{ij}}{d_j(L)}\right)x_j\right)^2 &\leq \left(\sum_{j=1}^n \frac{|d_j(G)-d_j(L)|}{d_j(G)d_j(L)}G_{ij}|x_j|\right)^2\\
        &\leq (2C)^2\left[2\left(\sum_{j\in S_i'} \frac{\varepsilon\cdot n\rho_n}{c'(c/2)\cdot n^2\rho_n^2}G_{ij}\right)^2 + 2\left(\sum_{j\in S_i}\frac{c''}{ n\rho_n}G_{ij}\right)^2\right]\\
        &\leq 8C^2\left[\left(\frac{\varepsilon}{c'(c/2)\cdot n\rho_n}d_i(G)\right)^2 + \frac{(c'')^2}{n^2\rho_n^2}|S_i|d_i(G)\right]
    \end{align*}
    where the final step is by Cauchy--Schwarz. Summing over $i$ gives
    \begin{align*}
        \frac{3}{n}\sum_{i=1}^n \left(\sum_{j=1}^n \left(\frac{G_{ij}}{d_j(G)} - \frac{G_{ij}}{d_j(L)}\right)x_j\right)^2&\leq 
        24C^2\left[\frac{\varepsilon^2}{(c')^2(c/2)^2}\frac{1}{n^3\rho_n^2}\sum_{i=1}^n d_i(G)^2 + (c'')^2\frac{1}{n^3\rho_n^2} \sum_{i=1}^n |S_i|d_i(G) \right]\\
        &\leq  24C^2\left[\frac{\varepsilon^2\sqrt{C'}}{(c')^2(c/2)^2} + (c'')^2\frac{1}{n^3\rho_n^2} \sum_{i=1}^n |S_i|d_i(G)\right]
    \end{align*}
    where the second inequality is by Lemma \ref{interpret}.
    
    We now deal with the summation in the final expression above. Note that
    \begin{align*}
        \norm{G-L}_\square &\geq \frac{1}{n^2}\left|\sum_{j\in S^+}\sum_{i=1}^n (G_{ij}-L_{ij})\right|\geq \frac{\varepsilon\rho_n}{n}\cdot |S^+|.
    \end{align*}
    Since $\norm{G-L}_\square/\rho_n\to 0$ almost surely by \eqref{eq:graph_conv}, we see that $|S^+|=o_p(n).$ The same argument gives $|S^-|=o_p(n)$ as well. Thus, we conclude
    \begin{align*}
        \sum_{i=1}^n |S_i| = \sum_{j\in S}d_j(G)\leq \sqrt{|S|\sum_{j=1}^n d_j(G)^2} = o_p(n^2\rho_n)
    \end{align*}
    where the last step is by Lemma \ref{interpret}. 
    
    Next, let $F=\{i: |S_i|\geq \varepsilon\cdot n\rho_n\}.$ Then, we get
    \begin{align*}
        |F|\cdot \varepsilon \cdot n\rho_n\leq \sum_{i=1}^n |S_i|\leq  o_p(n^2\rho_n)
    \end{align*}
    and thus $|F|=o_p(n).$ As a result, noting $|S_i|\leq d_i(G)$ and considering $i\in F$ and $i\in F'$ separately, we conclude that
    \begin{align*}
        \frac{1}{n^3\rho_n^2}\sum_{i=1}^n |S_i|d_i(G) &\leq \frac{1}{n^3\rho_n^2}\cdot (\varepsilon n\rho_n)\cdot \sum_{i=1}^n d_i(G) + \frac{1}{n^3\rho_n^2}\sum_{i\in F}d_i(G)^2\\
        &\leq \varepsilon\cdot (C')^{1/4} + \frac{1}{n^3\rho_n^2}\sqrt{|F|\sum_{i=1}^n d_i(G)^4}\\
        &\leq \varepsilon\cdot (C')^{1/4}+ \sqrt{\frac{|F|}{n}}\cdot \sqrt{C'}
    \end{align*}
    where we also used Lemma \ref{interpret}. Bringing everything together, we have shown that 
    \begin{align*}
        1_{E_n}\cdot \frac{3}{n}\sum_{i=1}^n \left(\sum_{j=1}^n \left(\frac{G_{ij}}{d_j(G)} - \frac{G_{ij}}{d_j(L)}\right)x_j\right)^2\leq C_1 \varepsilon^2 + C_2\cdot \varepsilon + C_3\cdot o_p(1)
    \end{align*}
    where $C_1,C_2,C_3>0$ are constants and the $o_p(1)$ term is uniformly bounded over $n.$ Taking expectations gives
    \begin{align*}
        M_1 &\leq C_1\varepsilon^2 + C_2\varepsilon + o(1).
    \end{align*}
    Taking the limsup as $n\to\infty,$ then letting $\varepsilon\to 0,$ we conclude that $M_1=o(1)$ as desired.

    \subsubsection{$M_2=o(1)$}
    Again, assume that the event 
    $E_n=\left\{\min\limits_{1\leq i\leq n}d_i(L)\geq c/2\cdot n\rho_n\right\}$
    holds. Fix any $\varepsilon > 0$ and let
    \begin{align*}
        S^+&= \left\{i\in [n]: 1_{E_n}\cdot \sum_{j=1}^n \left(\frac{G_{ij}}{d_j(L)} - \frac{L_{ij}}{d_j(L)}\right)x_j > \varepsilon\right\},\\
        S^-&= \left\{i\in [n]: 1_{E_n}\cdot \sum_{j=1}^n \left(\frac{G_{ij}}{d_j(L)} - \frac{L_{ij}}{d_j(L)}\right)x_j < -\varepsilon\right\}.
    \end{align*}
    Writing $S = S^+\cup S^-$ and considering $i\in S$ and $i\in S'$ separately, we get
    \begin{align*}
        \frac{3}{n}\sum_{i=1}^n \left(\sum_{j=1}^n \left(\frac{G_{ij}}{d_j(L)} - \frac{L_{ij}}{d_j(L)}\right)x_j\right)^2 &\leq \frac{3}{n}\left[n\cdot \varepsilon^2 + \frac{4C^2}{(c/2)^2n^2\rho_n^2}\sum_{i\in S}\left(d_i(G)+d_i(L)\right)^2\right]\\
        &\leq 3\varepsilon^2 + \frac{24C^2}{(c/2)^2}\frac{1}{n^3\rho_n^2}\sum_{i\in S}(d_i(G)^2+d_i(L)^2)\\
        &\leq 3\varepsilon^2 + \frac{24C^2}{(c/2)^2}\frac{1}{n^3\rho_n^2}\left(\sqrt{|S|\sum_{i=1}^n d_i(G)^4} + \sqrt{|S|\sum_{i=1}^n d_i(L)^4}\right).
    \end{align*}
    Taking expectations and applying Cauchy--Schwarz along with Lemmas \ref{interpret} and \ref{exp}, we get
    \begin{align*}
        M_2 &\leq 3\varepsilon^2 + \frac{24C^2}{(c/2)^2}\frac{1}{n^3\rho_n^2}\left(\sqrt{\E|S|C'n^5\rho_n^4} + \sqrt{\E|S|\E\left[\sum_{i=1}^n d_i(L)^4\right]}\right)\\
        &= 3\varepsilon^2 + C_1\cdot \sqrt{\frac{\E|S|}{n}}
    \end{align*}
    for some constant $C_1>0.$ 
    
    To conclude, we will show that $|S|=o_p(n).$ Note that 
    \begin{align*}
        1_{E_n\cdot }\frac{1}{n}\left|\sum_{i\in S^+} \sum_{j=1}^n \left(\frac{G_{ij}}{d_j(L)} - \frac{L_{ij}}{d_j(L)}\right)x_j \right| &=  1_{E_n}\cdot \frac{1}{n}\left| \sum_{j=1}^n \frac{x_j}{d_j(L)}\sum_{i\in S^+}\left(G_{ij}-L_{ij}\right) \right|\\
        &\leq \frac{2C}{(c/2)n^2\rho_n} \sum_{j=1}^n \left|\sum_{i\in S^+}\left(G_{ij}-L_{ij}\right) \right|\\
        &\leq \frac{4C}{c/2}\cdot\frac{\norm{G-L}_\square}{\rho_n}
    \end{align*}
    where the last inequality follows from partitioning the $j$'s based on whether the inner summation is positive or negative. Thus, we get
    \begin{align*}
        \frac{4C}{c/2}\cdot\frac{\norm{G-L}_\square}{\rho_n} &\geq 1_{E_n}\cdot \frac{1}{n}\left|\sum_{i\in S^+} \sum_{j=1}^n \left(\frac{G_{ij}}{d_j(L)} - \frac{L_{ij}}{d_j(L)}\right)x_j \right|\\
        &\geq \frac{1}{n}\sum_{i\in S^+} 1_{E_n}\cdot \sum_{j=1}^n \left(\frac{G_{ij}}{d_j(L)} - \frac{L_{ij}}{d_j(L)}\right)x_j\\
        &\geq \frac{|S^+|\varepsilon}{n}.
    \end{align*}
    Since $\rho_n^{-1}\norm{G-L}_\square\to 0$ almost surely, we conclude that $|S^+|=o_p(n)$. Proceeding identically, we get $|S^-|=o_p(n)$ as well.

    Combining everything, we have
    \begin{align*}
        M_2\leq 3\varepsilon^2 + C_1\cdot \sqrt{\E[o_p(1)]}
    \end{align*}
    where the $o_p(1)$ term is uniformly bounded over $n$. Taking the limsup as $n\to\infty,$ then $\varepsilon\to 0,$ we get $M_2=o(1)$ as desired.
    \subsubsection{$M_3=o(1)$}
    As in previous parts, assume that the event 
    $E_n=\left\{\min\limits_{1\leq i\leq n}d_i(L)\geq c/2\cdot n\rho_n\right\}$
    holds. Fix $\varepsilon > 0$ and let
    \begin{align*}
        S_i&=\{j\in [n]: L_{ij}=1, |x_j-y_j|>\varepsilon\},\\
        S&=\{j\in [n]: |x_j-y_j|>\varepsilon\}.
    \end{align*}
    Then, we see that
    \begin{align*}
        \frac{3}{n}\sum_{i=1}^n\left(\sum_{j=1}^n \frac{L_{ij}}{d_j(L)}(x_j-y_j)\right)^2 &\leq \frac{6}{n}\left[\left(\frac{1}{(c/2)^2 n^2\rho_n^2}\sum_{i=1}^n \varepsilon^2 d_i(L)^2\right)+\frac{16C^2}{(c/2)^2n^2\rho_n^2}\sum_{i=1}^n\left(\sum_{j\in S_i}L_{ij}\right)^2 \right]\\
        &\leq \frac{6\varepsilon^2}{(c/2)^2}\cdot \frac{1}{n^3\rho_n^2}\sum_{i=1}^n d_i(L)^2 + \frac{96C^2}{(c/2)^2}\frac{1}{n^3\rho_n^2}\sum_{i=1}^n |S_i|d_i(L).
    \end{align*}
    Since $\norm{\mathbf{x}-\mathbf{y}}_1\to 0$ almost surely, we see that $|S|=o_p(n)$ by similar arguments from before. Moreover,
    \begin{align*}
        \sum_{i=1}^n |S_i|= \sum_{j\in S}d_j(L)\leq \sqrt{|S|\sum_{j=1}^n d_j(L)^2}=o_p(n^2\rho_n)
    \end{align*}
    since we know that $\sum_{j=1}^n d_j(L)^2 = O_p(n^3\rho_n^2)$ by Lemma \ref{exp}.
    
    Let $F=\{i\in[n]: |S_i|\geq\varepsilon\cdot n\rho_n\}$ so that
    \begin{align*}
        \varepsilon\cdot n\rho_n |F|\leq \sum_{i=1}^n |S_i|\leq o_p(n^2\rho_n).
    \end{align*}
    This implies that $|F|=o_p(n).$ To conclude, note that 
    \begin{align*}
        \frac{1}{n^3\rho_n^2}\sum_{i=1}^n |S_i|d_i(L) &\leq \varepsilon\cdot \frac{1}{n^2\rho_n}\sum_{i=1}^n d_i(L) + \frac{1}{n^3\rho_n^2}\sum_{i\in F}d_i(L)^2
    \end{align*}
    where we used $|S_i|\leq d_i(L).$ Hence, we have shown that 
    \begin{align*}
        \frac{3}{n}\sum_{i=1}^n\left(\sum_{j=1}^n \frac{L_{ij}}{d_j(L)}(x_j-y_j)\right)^2 &\leq C_1\cdot \varepsilon^2\cdot \frac{1}{n^3\rho_n^2}\sum_{i=1}^n d_i(L)^2 + C_2\cdot \left(\varepsilon\cdot \frac{1}{n^2\rho_n}\sum_{i=1}^n d_i(L) + \frac{1}{n^3\rho_n^2}\sum_{i\in F}d_i(L)^2\right) \\
        &\leq  C_1\cdot \varepsilon^2\cdot \frac{1}{n^3\rho_n^2}\sum_{i=1}^n d_i(L)^2 + C_2\cdot \left(\varepsilon\cdot \frac{1}{n^2\rho_n}\sum_{i=1}^n d_i(L) + \frac{1}{n^3\rho_n^2}\sqrt{|F|\sum_{i=1}^n d_i(L)^4}\right) 
    \end{align*}
    on the event $E_n$ for some constants $C_1,C_2>0.$ Taking expectations and applying Lemma \ref{exp}, there are constants $C_1', C_2',C_3'>0$ such that
    \begin{align*}
        M_3 &\leq C_1'\cdot \varepsilon^2 + C_2'\cdot \varepsilon + C_3'\cdot \sqrt{\frac{\E|F|}{n}}.
    \end{align*}
    Taking the limsup as $n\to\infty,$ then $\varepsilon\to 0,$ we get $M_3=o(1)$ as desired. This concludes the proof.
\end{proof}

\subsection{Coupling the H\'{a}jek Estimators}
In this section, we state and prove an analogue of Lemma \ref{coupling_lemma} for the H\'{a}jek estimator.
\begin{lemma}\label{coupling_lemma_Haj}
    Under the same treatments $W_1,\dots, W_n,$ we have
\begin{align*}
    \hat\tau_n^{\Haj}(G_n^\varphi, v^\varphi_n) - \overline\tau_n(G_n^\varphi, v_n^\varphi)  = \hat\tau_n^{\Haj}(L_n^\sigma,\ell_n^\sigma) - \overline\tau_n(L_n^\sigma,\ell_n^\sigma) +  o_p(1/\sqrt{n}).
\end{align*}
\end{lemma}
\begin{proof}
    Recalling $\psi_i$ from Lemma \ref{move} and $\nu_i$ from Lemma \ref{move_Haj}, note that 
    \begin{align*}
          \frac{1}{n}\sum_{i=1}^n \Big(\nu_i(G_n^\varphi,v_n^\varphi) - \nu_i(L_n^\sigma,\ell_n^\sigma)\Big)(W_i-\pi) =  \frac{1}{n}\sum_{i=1}^n \Big(\psi_i(G_n^\varphi,v_n^\varphi) - \psi_i(L_n^\sigma,\ell_n^\sigma)\Big)(W_i-\pi) + D_3
    \end{align*}
    where 
    \begin{align*}
        D_3 &=  \frac{1}{n}\sum_{i=1}^n\frac{1}{n}\sum_{j=1}^n  \Bigg[\frac{\ell_n^\sigma(j,1,\pi)-v_n^\varphi(j,1,\pi)}{\pi} + \frac{\ell_n^\sigma(j,0,\pi) - v_n^\varphi(j,0,\pi)}{1-\pi}\Bigg](W_i-\pi)\\
        &= \frac{1}{n}\sum_{j=1}^n  \Bigg[\frac{\ell_n^\sigma(j,1,\pi)-v_n^\varphi(j,1,\pi)}{\pi} + \frac{\ell_n^\sigma(j,0,\pi) - v_n^\varphi(j,0,\pi)}{1-\pi}\Bigg]\cdot \Bigg[\frac{1}{n}\sum_{i=1}^n (W_i-\pi)\Bigg].
    \end{align*}
    By Lemmas \ref{coupling_lemma}, \ref{move}, \ref{move_Haj}, it suffices to show that $D_3 = o_p(1/\sqrt{n}).$ Note that 
    \begin{align*}
       &\left|\frac{1}{n}\sum_{j=1}^n  \Bigg[\frac{\ell_n^\sigma(j,1,\pi)-v_n^\varphi(j,1,\pi)}{\pi} + \frac{\ell_n^\sigma(j,0,\pi) - v_n^\varphi(j,0,\pi)}{1-\pi}\Bigg]\right|\\
       &\leq \frac{\norm{\ell_n^\sigma(\cdot,1,\pi) - v_n^\varphi(\cdot,1,\pi)}_1}{\pi} + \frac{\norm{\ell_n^\sigma(\cdot,0,\pi) - v_n^\varphi(\cdot,0,\pi)}_1}{1-\pi} \\
       &= o_p(1)
    \end{align*}
    where the last step is by \eqref{eq:pot_conv}. Since $n^{-1}\sum_{i=1}^n (W_i-\pi)=O_p(1/\sqrt{n}),$ we obtain $D_3 = o_p(1/\sqrt{n})$ and conclude the proof.
\end{proof}
\section{Distribution of Estimators after Relabeling}
In Section \ref{decompose_app}, we proved Lemmas \ref{coupling_lemma} and \ref{coupling_lemma_Haj} to get
\begin{align*}
    \hat\tau_n^{\HT}(G_n^\varphi, v^\varphi_n) - \overline\tau_n(G_n^\varphi, v_n^\varphi)  &= \hat\tau_n^{\HT}(L_n^\sigma,\ell_n^\sigma) - \overline\tau_n(L_n^\sigma,\ell_n^\sigma) +  o_p(1/\sqrt{n}),\\
    \hat\tau_n^{\Haj}(G_n^\varphi, v^\varphi_n) - \overline\tau_n(G_n^\varphi, v_n^\varphi)  &= \hat\tau_n^{\Haj}(L_n^\sigma,\ell_n^\sigma) - \overline\tau_n(L_n^\sigma,\ell_n^\sigma) +  o_p(1/\sqrt{n}).
\end{align*}
In order to relate these results to $\widehat\tau_n(G_n,v_n) - \overline{\tau}_n(G_n,v_n)$, we prove the following general lemma.

\begin{lemma}\label{permutation}
        Let $(A_n,u_n)$ describe a (possibly random) finite population under anonymous interference. Let $\xi_n\in S_n$ denote a (possibly random) permutation. If $(A_n,u_n,\xi_n)$ is independent of the treatments $\{W_i\}_{i=1}^n$, then we have
        \begin{align*}
            \hat\tau_n^{\HT}(A_n^{\xi}, u_n^\xi) - \overline\tau_n(A_n^{\xi}, u_n^\xi) &\ed \hat\tau_n^{\HT}(A_n, u_n) - \overline\tau_n(A_n, u_n).\\
            \hat\tau_n^{\Haj}(A_n^{\xi}, u_n^\xi) - \overline\tau_n(A_n^{\xi}, u_n^\xi) &\ed \hat\tau_n^{\Haj}(A_n, u_n) - \overline\tau_n(A_n, u_n).
        \end{align*}
\end{lemma}

\begin{proof}[Proof of Lemma~\ref{permutation}]
Given a graph $A_n,$ we again write $d_i(A)$ instead of $d_i(A_n)$. We begin with the HT estimator. Recall that
\begin{align*}
    \hat\tau_n^{\HT}(A_n, u_n) &= \frac{1}{n}\sum_{i=1}^n \left(\frac{W_i}{\pi}-\frac{1-W_i}{1-\pi}\right)u_{n,i}\left(W_i, \frac{\sum_{j\sim i}W_j}{d_i(A)}\right).
\end{align*}
For deterministic $(B,b,w),$ define
\begin{align*}
    H(B,b,w)= \frac{1}{n}\sum_{i=1}^n \left(\frac{w_i}{\pi}-\frac{1-w_i}{1-\pi}\right)b_i\left(w_i, \frac{\sum_{j\sim i}w_j}{d_i(B)}\right)
\end{align*}
so that $\hat\tau_n^{\HT}(A_n,u_n) = H(A_n,u_n,W).$ We check that $H$ is permutation-invariant. Fix any $\eta\in S_n.$ Then, we get
\begin{align*}
   H(B^\eta,b^\eta,w^\eta) &= \frac{1}{n}\sum_{i=1}^n \left(\frac{w_{\eta(i)}}{\pi}-\frac{1-w_{\eta(i)}}{1-\pi}\right)b_{\eta(i)}\left(w_{\eta(i)}, \frac{\sum_{j\sim_{B^\eta} i}w_{\eta(j)}}{d_{i}(B^\eta)}\right)\\
    &= \frac{1}{n}\sum_{i=1}^n \left(\frac{w_{\eta(i)}}{\pi}-\frac{1-w_{\eta(i)}}{1-\pi}\right)b_{\eta(i)}\left(w_{\eta(i)}, \frac{\sum_{j: \eta(j)\sim_{B} \eta(i)}w_{\eta(j)}}{d_{\eta(i)}(B)}\right)\\
    &= \frac{1}{n}\sum_{i=1}^n \left(\frac{w_{i}}{\pi}-\frac{1-w_{i}}{1-\pi}\right)b_{i}\left(w_{i}, \frac{\sum_{j: \eta(j)\sim_{B} i}w_{\eta(j)}}{d_{i}(B)}\right)\\
    &= \frac{1}{n}\sum_{i=1}^n \left(\frac{w_{i}}{\pi}-\frac{1-w_{i}}{1-\pi}\right)b_{i}\left(w_{i}, \frac{\sum_{j \sim_{B} i}w_{j}}{d_i(B)}\right)\\
    &= H(B,b,w)
\end{align*}
as desired. Similarly, recall that 
\begin{align*}
   \overline\tau_n(A_n,u_n) &=\E\left[\frac{1}{n}\sum_{i=1}^n u_{n,i}\left(1, \frac{\sum_{j\sim i}W_j}{d_i(A)}\right)-u_{n,i}\left(0, \frac{\sum_{j\sim i}W_j}{d_i(A)}\right)\mid A_n,u_n\right].
\end{align*}
Since $(A_n,u_n)$ is independent of $\{W_i\}_{i=1}^n,$ we can write $\overline\tau_n = h(A_n,u_n)$ where 
\begin{align*}
    h(B,b)= \E\left[\frac{1}{n}\sum_{i=1}^n b_i\left(1, \frac{\sum_{j\sim i}W_j}{d_i(B)}\right)-b_i\left(0, \frac{\sum_{j\sim i}W_j}{d_i(B)}\right)\right]
\end{align*}
for deterministic $(B,b).$ By the same relabeling computation as above, one can check that $h$ is also permutation invariant: For any fixed $\eta\in S_n,$ we have $h(B^\eta,b^\eta) = h(B,b).$

Therefore, we can write
\begin{align*}
    \hat\tau_n^{\HT}(A_n,u_n) - \overline{\tau}_n(A_n,u_n) &= H(A_n,u_n,W) - h(A_n,u_n),\\
    \hat\tau_n^{\HT}(A_n^\xi,u_n^\xi) - \overline{\tau}_n(A_n^\xi,u_n^\xi) &= H(A_n^\xi,u_n^\xi,W) - h(A_n^\xi,u_n^\xi)\\
    &= H(A_n,u_n,W^{\xi^{-1}})-h(A_n,u_n).
\end{align*}
Conditional on $(A_n,u_n,\xi_n),$ we know that
\begin{align*}
    (W_1,\dots, W_n)\ed \left(W_{\xi^{-1}(1)},\dots, W_{\xi^{-1}(n)}\right).
\end{align*}
Hence, conditional on $(A_n,u_n,\xi_n),$ we see that
\begin{align*}
    \hat\tau_n^{\HT}(A_n,u_n) - \overline{\tau}_n(A_n,u_n) \ed \hat\tau_n^{\HT}(A_n^\xi,u_n^\xi) - \overline{\tau}_n(A_n^\xi,u_n^\xi).
\end{align*}
In particular, the result holds unconditionally as well. 

To conclude, the argument for the H\'{a}jek estimator is identical. For deterministic $(B,b,w),$ define
\begin{align*}
    J(B,b,w)= \frac{\sum_{i=1}^n w_i b_i\left(w_i,\frac{\sum_{j\sim i}w_j}{d_i(B)}\right)}{\sum_{i=1}^n w_i} - \frac{\sum_{i=1}^n (1-w_i) b_i\left(w_i,\frac{\sum_{j\sim i}w_j}{d_i(B)}\right)}{\sum_{i=1}^n (1-w_i)}
\end{align*}
so that $\hat\tau_n^{\Haj}(A_n,u_n) = J(A_n,u_n,W).$ It suffices to show that $J$ is permutation-invariant. Fix any $\eta\in S_n.$ Then, we get
\begin{align*}
   &J(B^\eta,b^\eta,w^\eta) \\
   &=
   \frac{\sum_{i=1}^n w_{\eta(i)}b_{\eta(i)}\left(w_{\eta(i)}, \frac{\sum_{j\sim_{B^\eta} i}w_{\eta(j)}}{d_{i}(B^\eta)}\right)}{\sum_{i=1}^n w_{\eta(i)}} - \frac{\sum_{i=1}^n \left(1-w_{\eta(i)}\right)b_{\eta(i)}\left(w_{\eta(i)}, \frac{\sum_{j\sim_{B^\eta} i}w_{\eta(j)}}{d_{i}(B^\eta)}\right)}{\sum_{i=1}^n (1-w_{\eta(i)})}.
\end{align*}
Since $\sum_{i=1}^n w_{\eta(i)} = \sum_{i=1}^n w_i,$ it remains to deal with the numerators. Note that
\begin{align*}
    \sum_{i=1}^n w_{\eta(i)}b_{\eta(i)}\left(w_{\eta(i)}, \frac{\sum_{j\sim_{B^\eta} i}w_{\eta(j)}}{d_{i}(B^\eta)}\right) &= 
    \sum_{i=1}^n w_{\eta(i)}b_{\eta(i)}\left(w_{\eta(i)}, \frac{\sum_{j: \eta(j)\sim_{B} \eta(i)}w_{\eta(j)}}{d_{\eta(i)}(B)}\right)\\
    &=\sum_{i=1}^n w_ib_{i}\left(w_{i}, \frac{\sum_{j: \eta(j)\sim_{B} i}w_{\eta(j)}}{d_{i}(B)}\right)\\
    &=\sum_{i=1}^n w_ib_{i}\left(w_{i}, \frac{\sum_{j \sim_{B} i}w_{j}}{d_i(B)}\right).
\end{align*}
By the same argument for the other numerator, we conclude that $J(B^\eta,b^\eta,w^\eta) = J(B,b,w)$ as desired and conclude the proof.
\end{proof}

\section{Proof of Theorem \ref{Main} and Corollary \ref{clarify}}
We are now ready to prove Theorem \ref{Main} and Corollary \ref{clarify}.
\begin{proof}[Proof of Theorem \ref{Main}]
Since $(G_n,v_n,\varphi_n)$ is deterministic while all randomness in $(L_n,\ell_n,\sigma_n)$ is independent of the treatments, Lemma \ref{permutation} is applicable to both finite populations. Thus, Proposition \ref{wager_main} applies to the relabeled random populations $(L_n^\sigma,\ell_n^\sigma)$, and combining this with Lemma \ref{coupling_lemma} and Slutsky's theorem gives
     \begin{align*}
    \sqrt{n}\left(\hat\tau_n^{\HT}(G_n^{\varphi}, v_n^\varphi) - \overline\tau_n(G_n^{\varphi}, v_n^\varphi)\right)\limd N(0,V^{\HT}_{\L,\ell,\pi}).
    \end{align*}
    A final application of Lemma \ref{permutation} removes the relabeling and gives the claimed result for $(G_n,v_n).$ This concludes the proof of Theorem \ref{Main} for the HT estimator. 
    
    Similarly, Theorem 4 of \cite{Wager} gives
    \begin{align*}
    \sqrt{n}(\hat\tau_n^{\Haj}(L_n,\ell_n) - \overline\tau_n(L_n,\ell_n))\limd N\left(0,V^{\Haj}_{\L,\ell,\pi}\right).
    \end{align*}
Applying Lemma \ref{permutation} to the above superpopulation CLT, then using Lemma \ref{coupling_lemma_Haj} along with Slutsky's theorem, we obtain
 \begin{align*}
\sqrt{n}\left(\hat\tau_n^{\Haj}(G_n^{\varphi}, v_n^\varphi) - \overline\tau_n(G_n^{\varphi}, v_n^\varphi)\right)\limd N(0,V^{\Haj}_{\L,\ell,\pi}).
\end{align*}
Finally, Lemma \ref{permutation} removes the relabeling and gives the claimed result for $(G_n,v_n).$ This concludes the proof.
\end{proof}

\begin{proof}[Proof of Corollary \ref{clarify}]
    As mentioned in the main text, it suffices to show that $\{L_n\}_{n=1}^\infty$ almost surely satisfies Condition \ref{Deg}. This is given in Lemma \ref{concentration_deg} and we conclude the proof.
\end{proof}

\section{Simulation Details}
All code for our experiments is available at \url{https://github.com/bryan1292929/Network-interference}. We describe how to compute the limiting variances when 
\begin{align*}
    \L(u,v)&=1_{u\neq v}\cdot 1_{u+v>1},\\
    \ell(t,w,x)&= t + (1+4t)w + (2+2t)x + 5x^2 + 4wx,
\end{align*} 
and $\pi=0.5.$ Recall that 
\begin{align*}
        \mathcal{R}_i &= \frac{\ell(U_i,1,\pi)}{\pi} + \frac{\ell(U_i,0,\pi)}{1-\pi},\qquad
        \mathcal{Q}_i = \E\left[\left.\frac{\mathcal{L}(U_i,U_j)\left(\ell'(U_j,1,\pi) - \ell'(U_j,0,\pi)\right)}{\int_{[0,1]}\L(u,U_j)\,du}\right| U_i\right].
    \end{align*}
In our case, we have $\ell'(t,1,\pi)-\ell'(t,0,\pi)=4.$ Moreover, we know that $\int_0^1 \L(u,v)\,du=v.$ Hence, we get
\begin{align*}
    \mathcal{Q}_i = 4\cdot \E\left[\frac{\L(U_i,U_j)}{U_j}\mid U_i\right] &= 4\int_{0}^1 \frac{\L(U_i,x)}{x}\,dx \\
    &= 4\cdot \int_{1-U_i}^1\frac{1}{x}\,dx \\
    &= -4\ln(1-U_i).
\end{align*}
As the asymptotic variances are given by
\begin{align*}
        V^{\HT}_{\L,\ell,\pi} &= \pi(1-\pi)\E[(\mathcal{R}_i + \mathcal{Q}_i)^2],\qquad
        V^{\Haj}_{\L,\ell,\pi} = \pi(1-\pi)\left(\Var(\mathcal{R}_i+\mathcal{Q}_i) + \left(\E[\mathcal{Q}_i]\right)^2\right),
    \end{align*} 
    we can use Monte Carlo sampling to estimate $V^{\HT}_{\L,\ell,\pi},V^{\Haj}_{\L,\ell,\pi}$ by sampling many $U_i$'s.
\end{document}